\documentclass{amsart}
\usepackage{fullpage}
\usepackage[latin1]{inputenc}
\usepackage{amssymb}
\usepackage{amsmath}
\usepackage{enumerate}
\usepackage{epsfig,color,caption}
\usepackage{longtable,graphics,multirow,ulem}
\usepackage{graphicx,color}
\usepackage[curve]{xypic} 
\usepackage{hyperref}
\usepackage{graphicx}
\usepackage[dvipsnames]{xcolor}
\usepackage{colonequals, tikz}
\usepackage{tikz-cd}

\usepackage{color}

\usepackage{longtable,graphics,multirow,ulem}
\usepackage{tikz}
\usepackage{pgf,tikz}
\usetikzlibrary{arrows}
\usepackage[all]{xy}

\newcounter{nootje}
\newtheorem{theorem}{Theorem}[section]
\newtheorem{lemma}[theorem]{Lemma}
\newtheorem{proposition}[theorem]{Proposition}
\newtheorem{corollary}[theorem]{Corollary}

\theoremstyle{definition}
\newtheorem{definition}[theorem]{Definition}
\newtheorem{convention}[theorem]{Convention}
\newtheorem{remark}[theorem]{Remark}

\newtheorem{notation}[theorem]{Notation}
\newtheorem{example}[theorem]{Example}

\numberwithin{equation}{section}
\newcommand{\rk}{\mbox{rank}}

\newcommand{\ra}{\rightarrow}

\newcommand{\Z}{\mathbb{Z}}
\newcommand{\Q}{\mathbb{Q}}

\newcommand{\N}{\mathbb{N}}

\renewcommand{\O}{\mathcal{O}}

\DeclareMathOperator{\Gal}{Gal}
\DeclareMathOperator{\Pic}{Pic}
\DeclareMathOperator{\NS}{NS}
\DeclareMathOperator{\MW}{MW}

\usepackage{fancyhdr}
\makeatletter
\def\blfootnote{\xdef\@thefnmark{}\@footnotetext}

\author{Alice Garbagnati}
\address{Dipartimento di Matematica, Univ. Statale di Milano, Milan, Italy}
\email{alice.garbagnati@unimi.it}
\urladdr{ https://sites.google.com/site/alicegarbagnati/}
\author{Cec\'ilia Salgado}
\address{Bernoulli Institute, University of Groningen, the Netherlands}
\email{c.salgado@rug.nl}
\urladdr{ https://www.math.rug.nl/algebra/Main/salgado}

\title[Rank jumps and Multisections on K3s ]{Rank jumps and Multisections of elliptic fibrations on K3 surfaces}
	
\begin{document}


\maketitle
\begin{abstract} We consider the countably many families $\mathcal{L}_d$, $d\in\N_{\geq 2}$, of K3 surfaces admitting an elliptic fibration with positive Mordell--Weil rank. We prove that the elliptic fibrations on the very general member of these families have the potential  Mordell--Weil rank jump property for $d\neq 2,3$ and moreover the  Mordell--Weil rank jump property for $d\equiv 3\mod 4$, $d\neq 3$. We provide explicit examples and discuss some extensions to subfamilies. The result is based on the geometric interaction between the (potential) Mordell--Weil rank jump property and the presence of special multisections of the fibration. \end{abstract}

\section{Introduction}

The arithmetic of K3 surfaces has been intensively studied and witnessed great progress over the past 25 years. Two prominent and intertwined topics are (potential) Zariski density of rational points on K3 surfaces over non-algebraically closed fields, and elliptic fibrations, in particular, the study of Mordell--Weil rank jumps among fibers of elliptic fibrations on K3 surfaces.

In this article, we study (potential) Mordell--Weil rank jumps for elliptic K3 surfaces. Our main result is the following.

\begin{theorem}\label{thm:rkjump}
Let $X$ be a K3 surface defined over a number field such that $\Pic(X)=U \oplus \langle -2d\rangle$, for $d\in\mathbb{N}$, or equivalently a K3 surface defined over $k$ admitting an elliptic fibration (possibly over a  field extension of $k$) and with Picard number 3. Assume that $d\neq 2,3$. Then each elliptic fibration on $X$ has the potential rank jump property and, moreover, if $d\equiv 3\mod 4$, then it has the rank jump property over its field of definition. 
\end{theorem}

The innovation of our contribution is two-fold. Firstly, Theorem \ref{thm:rkjump} presents the first instance of rank jump in a general context that does not involve a family of multisections but relies on the presence of a specific type of multisection, namely a saliently ramified multisection as introduced by Bogomolov and Tschinkel (\cite[Definition 2.3]{BT2000}) and adapted to the context of Mordell--Weil rank jumps by one of the authors with Pasten in \cite{PastenSalgado}. Secondly, to produce a saliently ramified multisection, we combine lattice theoretic and geometric methods, a technique used for the first time in that context. This geometric description allows one to control the field over which the jump property appears, and therefore we describe countable many families of K3 surfaces which have the rank jump property, and not just the potential one.

In what follows, we give a brief account on developments on (potential) density of rational points on K3 surfaces, and rank jumps on elliptic surfaces, leading to the motivation behind Theorem \ref{thm:rkjump}.

Throughout the text, we let $k$ be a number field.

\subsection*{(Potential) Zariski density} 
In \cite[Theorem 1.1]{BT}, Bogomolov and Tschinkel show that the rational points on K3 surfaces defined over $k$ that admit either a genus 1 fibration or an infinite automorphism group are potentially dense in the Zariski topology, i.e., there exists a finite extension $l/k$ such that the $l$-points of the K3 surface are not contained in a finite union of proper subvarieties. One of the tools they introduce is the concept of saliently ramified multisections (\cite[Definition 2.3]{BT2000}), which are irreducible curves in the surface for which the restriction of the fibration map is of non-zero degree and ramified over at least one smooth fiber.

The relevance of saliently ramified rational or genus 1 multisections in the proof of potential density of rational points stems from the fact that they are non-torsion multisections on the associated jacobian fibration (\cite[Proposition 4.4]{BT}, \cite[Proposition 2.4]{BT2000}). After a base change by such a multisection, the new fibration has positive Mordell--Weil rank and a specialization argument applied to the points on the image of the restriction map yields the desired Zariski density over the field of definition of the multisection.

\subsection*{Rank jumps} Given an elliptic fibration $\pi: X \rightarrow C$ on a smooth projective algebraic surface $X$ over $k$, denote by $r$ the rank of the Mordell--Weil group of  $k$-sections of $\pi$, i.e., the rank of its generic fiber, and $r_t$ the rank of the Mordell--Weil group of the fiber $\pi^{-1}(t)$. We say that $\pi$ has the (Mordell--Weil) \textit{rank jump property} if the set $\mathcal{R}(X, \pi, k):=\{t \in C(k); r_t > r \}$ is infinite. By a conjecture of Silverman (\cite[p. 556]{Silverman85}), it is expected that every non-trivial elliptic fibration over $\mathbb{P}^1$ over $k$ has the (Mordell--Weil) rank jump property. In particular, this is expected for elliptic K3 surfaces. In the remaining of the article we omit the reference to Mordell--Weil and refer to the \textit{rank jump property}.

An immediate consequence of the aforementioned work by Bogomolov and Tschinkel is that an elliptic K3 surface with generic rank 0 Mordell--Weil group has the potential rank jump property, namely there is a finite extension $l/k$ such that $\mathcal{R}(X,\pi, l)$ is infinite. 

In \cite{SalgadoK3}, the second-named author shows that certain K3 surfaces with two elliptic fibrations have the potential rank jump property. Recently, in a collaboration with Pasten (\cite{PastenSalgado}), they show a stronger result for certain doubly elliptic K3 surfaces, namely that $\mathcal{R}(X,\pi, l) \subset C \simeq \mathbb{P}^1(l)$ is not thin in the sense of Serre (\cite[Section 9.1]{Serre}), for some finite extension $l/k$. 

In this context, the following questions arise:

\subsubsection*{Question 1:} Are there other examples of elliptic K3 surfaces that have the (potential) rank jump property, i.e., the rank jump property over (a finite extension of) the ground field? In particular, are there examples with a non-finite Mordell--Weil group and that are not doubly elliptic?
\indent\par
\indent\par
The main tool in \cite{PastenSalgado} is a modification of Bogomolov-Tschinkel's saliently ramified multisections that allows one to control the subset of the base of the fibration where the restriction morphism is \'etale. The key observation which actually holds in the context of \cite{BT} is that salient multisections not only are non-torsion but, moreover, are independent of the sections induced by the Mordell--Weil group of the original elliptic fibration in the base change (\cite[Lemma 3.3]{PastenSalgado}). 
This gives us a potential tool to answer Question 1, namely by addressing the following:

\subsubsection*{Question 2:} Given a K3 surface with a genus 1 fibration. Under which conditions are there saliently ramified rational multisections or singular genus 1 multisections? 

\indent\par

In the previous question we allow both singular and smooth rational multisections, but we restrict to singular multisections of genus 1 since if there exists a smooth genus 1  multisection, then the surface is automatically doubly elliptic.

\subsection*{Main techniques.} 
 Given the discussion above, it is natural to search for a better understanding of multisections of elliptic, or more generally genus 1, fibrations on K3 surfaces, in particular with respect to the property of being saliently ramified. In what follows, we denote by $\pi: X \rightarrow \mathbb{P}^1$ a genus 1 fibration on a K3 surface $X$. We show the following lattice theoretic results.
\begin{itemize}
    \item[1a)] The generic member in the 18-dimensional family of $\langle 2\rangle \oplus \langle -2 \rangle$--polarized K3 surfaces admits a saliently ramified smooth rational bisection (Proposition \ref{prop: F rat smooth bisection}).
    \item[1b)]  The generic member in the 18-dimensional family of $U(2)$--polarized K3 surfaces admits a saliently ramified smooth bisection of genus 1 (Proposition \ref{prop: F g=1 smooth bisection}).
    \item[1c)] If a genus 1 fibration on a K3 surface admits a smooth bisection and it has Picard number 2, then it admits a smooth bisection of genus 0 or 1 (Corollary \ref{cor:bisections}).
    \item[1d)] Any generic member of the families in (1a) and (1b) admits a singular salient rational bisection (Remarks \ref{rem: singular rational salient bisection case g=0}(c) and \ref{rem: singular rational salient bisection case g=1}). 
    \item[2a)] The generic member of the family of $\Lambda_d:=U\oplus \langle -2d \rangle$-polarized K3 surfaces, for $d\geq 4$, admits a smooth bisection, which is of genus 0 if $d$ is odd, and genus 1 if $d$ is even. This bisection is saliently ramified. In particular, there are no generic members in the families that admit simultaneously  smooth rational and smooth genus 1 bisections (Theorem \ref{theorem:Picard3}).
    \item[2b)] The generic member of any of the families of $\Lambda_d$-polarized K3 surfaces, with $d\leq 3$, does not admit a smooth multisection of genus 0 or 1 (Corollary \ref{cor: no multisections for low d}).
 
\end{itemize}

In particular, 2a) implies that each elliptic fibration on a generic member of the family of $\Lambda_d$-polarized K3 surfaces with $d\geq 5$, $d\equiv 1\mod 2$ admits a saliently ramified smooth rational bisection and hence satisfies the potential rank jump property (Theorem \ref{thm:rkjump}). This yields three new 17-dimensional families of examples of K3 surfaces that satisfy the potential rank jump property (Theorem \ref{thm:rkjump}). In particular, this answers \textbf{Questions 1} and \textbf{2} for the generic members of three 17-dimensional families of elliptic K3 surfaces. Moreover, the geometric construction of the bisection allows us to control the field on which it has rational points, and so to prove the rank jump property over the field of definition if $d\equiv 3\mod 4$.

Theorem \ref{thm:rkjump} states the potential rank jump property for all the generic members of the families of elliptic K3 surfaces with a non-trivial Mordell--Weil group, with only two exceptions, and the rank jump property for countable many families among them.
If $d\geq 14$ or $d=9$ the potential rank jump property for surfaces $X$ with $\Pic(X)\simeq \Lambda_d$ follows from the fact that these surfaces are doubly elliptic, but we are also able to discuss the property of the field extension needed to have the rank jump.

We discuss the analogous properties for generic members of certain sub-families of dimension 16.

We provide explicit equations of K3 surfaces with Picard rank 3 which have the potential rank jump property; to compute their Picard numbers we combine the known techniques based on Tate's conjecture with a more geometric and new approach that rely on the presence of an elliptic fibration. 
\subsection*{Organization of the text}

In Section \ref{section: preliminaries}, we cover the background on elliptic fibrations, multisections, base-change and rank jumps. We also cover the necessary background on K3 surfaces, lattice polarizations and their interplay with genus 1 fibrations. Section \ref{subsec:bisections lattice} focuses on the study of bisections via lattice theory and related geometric constructions. In particular, we show that elliptic K3 surfaces with Picard number 3 always admit a smooth bisection which is of genus 0 or 1,  with the exception of 3 families (Theorem \ref{theorem:Picard3}). We show, moreover, that for a generic such K3 the aforementioned bisection is saliently ramified. In order to prove that these bisections are saliently ramified, we give a realization of such a K3 surface as a double cover of the projective plane branched on a nodal sextic (Prop. \ref{prop: Fratx cover of P2}) or of the quadric surface branched on a curve of bidegree $(4,4)$ (Prop. \ref{prop: Fratx cover of P1xP1}), depending of the genus of the bisection and then we describe the elliptic fibration and the bisection in this model. Section \ref{section: higher degree} presents conditions  on the polarization to guarantee the existence of a smooth rational curve which is a saliently ramified higher-degree multisection and gives a brief account of the geometric realization of K3 surfaces with such polarizations. Section \ref{section: rank jump} presents our contribution to the study of rank jumps on elliptic K3 surfaces. More precisely, given $d>3$ and a  $U\oplus \langle 2d\rangle$-polarized K3 surface, we show in Theorems \ref{thm: rank jump odd d} and \ref{thm: rank jump_even} that it has the potential rank jump property and in Theorem \ref{thm: rank jump odd d} that it has the rank jump property if $d\equiv 3\mod 4$. In particular, Theorem \ref{thm: rank jump odd d} answers Question 1. We also extend the previous results to generic members of many subfamilies of codimension 1. Finally, Section \ref{section: examples} analyzes in detail examples of a K3 surface that admits a unique elliptic fibration with positive Mordell--Weil rank that was elusive to previous results and to which our results apply. We also consider an example in which the rank jump property can be proved by considering a rational multisection of degree 3, instead of a bisection. This provides an application of the results of Section \ref{section: higher degree}.

\section{Preliminaries}\label{section: preliminaries}

Let $k$ be a number field. In what follows, curves and surfaces over a field are always assumed to be irreducible. 

All divisors considered are on an algebraic surface, i.e., are formal sums of curves. Let $X$ be a smooth, projective and geometrically integral surface over $k$. The intersection form is a bilinear pairing on $\Pic(X)$, the group of divisors modulo linear equivalence. Given two divisors $D$ and $E$ in $\Pic(X)$, we denote their intersection by $DE$. The divisor classes that are algebraically equivalent to 0 over $\bar{k}$ form a group denoted by $\Pic^0(X)$. The quotient $\Pic (X)/ \Pic^0(X)$ is called the N\'eron--Severi group of $X$. Its rank is called the Picard number of $X$. For K3 surfaces, $\Pic^0$ is trivial, so $\NS(X) \simeq \Pic(X)$.

\subsection{Genus 1 fibrations and Elliptic surfaces}

We cover the necessary background for this text on genus 1 and elliptic fibrations. The interested reader can consult \cite{Mi} for an account of the basic theory.

\begin{definition}
Given a smooth projective curve $C$ over $k$, and a smooth projective surface $X$ over $k$, we say that a morphism $\pi:X\to C$ is an \textit{elliptic fibration} if it is a genus 1 fibration with a distinguished section defined over $k$ (these are sometimes called \textit{jacobian fibrations}). The triple $(X,\pi, C)$, or simply $X$ when $\pi$ and $C$ are clear from the context, is called an \textit{elliptic surface}.
\end{definition}

The set of $k$-sections of $\pi$ is an abelian group which we denote by $\MW(X,\pi,k)$. We often omit $k$, if the field of definition is clear from the context. The fibration $\pi:X\to C$ is called \textit{trivial} if its generic fiber is isomorphic to $E\otimes_k k(C)$ for some elliptic curve $E$ over $k$.

\begin{notation}
In what follows,  $X$ always denotes a surface with a genus 1 fibration, not necessarily elliptic. We reserve the letter $\pi$ for genus 1 fibrations, keeping in mind that there might be more than one such fibration on $X$. We denote by $U_\pi$ the Zariski open of good reduction of $\pi$.
\end{notation}

\subsubsection{Multisections of genus 1 fibrations}
\begin{definition}
Given a  genus 1 fibration $\pi: X \rightarrow C$, a geometrically integral curve $M \subset X$ is called a \textit{multisection} of $\pi$ if $\pi|_{M}: M \rightarrow C$ is finite and of degree larger than 1.
\end{definition}

\begin{remark}
    Our definition is slightly different from that in \cite[Def. 3.2.]{BT}. We assume that the degree of the restriction of the fibration is larger than 1, while \cite{BT} assumes that it is non-zero. So, in our definition, sections are not multisections. This modification is natural in our settings. Indeed, we are interested in the ramification of the restriction of the fibration of the multisection (Def. \ref{def:salient}) which clearly only makes sense for maps of degree at least 2.
\end{remark}

The presence of low genus multisections for elliptic fibrations has been explored both in the context of (potential) density of rational points on algebraic surfaces and of rank jump for elliptic fibrations. Indeed, in \cite{BT2000} the authors observe that a torsion multisection $T$ is such that $\pi|_T: T \rightarrow C$ is \'etale over $U_\pi$. This motivated the following definition.

\begin{definition}\label{def:salient}(\cite[Definition 2.3]{BT2000})
A multisection $M$ of $\pi$ is called \textit{saliently ramified} if $\pi|M: M \rightarrow C$ is ramified above at least one smooth fiber of $\pi$. 
\end{definition}
Building on the observation on torsion multisections above, they show that saliently ramified multisections yield non-torsion sections after base change (\cite[Prop. 2.10]{BT2000}).

\begin{remark}\label{rem:salient}
Given $M$ as in Def. \ref{def:salient}, the map $\pi|_M: M \rightarrow C$ is ramified at the points $t\in C$ such that the intersection of $M$ and $\pi^{-1}(t)$ is non-reduced. For fibrations for which all fiber components are reduced, these correspond to non-transverse intersections.
\end{remark}

In situations where a non-torsion multisection $M$ has infinitely many rational points over a given field $l$, one can conclude that the rational points are Zariski dense in $X$ over $l$.  Indeed, by a specialization argument on the base change $X\times_C M$ followed by considering the fibers above the points in the image of $\pi|_M: M \rightarrow C$, there are infinitely many elliptic curves with positive Mordell--Weil rank over $l$. 

\subsection{Rank jump and base change}\label{subsection: base change}
Let $(X,\pi, C)$ be an elliptic surface defined over $k$. Let $l/k$ be a finite extension of $k$.

\begin{notation}
 For $t\in C$, we let $X_t=\pi^{-1}(t)$, i.e., the fiber of $\pi$ above $t$.  Denote by $\mathcal{R}(X,\pi, l)$ the set 
\[
\mathcal{R}(X,\pi, l):= \{t \in C(l); \mathrm{rank } \,X_t(l) > \mathrm{rank }\, \mathrm{MW }(X, \pi,l) \}.
\]
\end{notation}

\begin{definition}
We say that $\pi$ has the \textit{rank jump property} (resp. the \textit{potential rank jump property}), if $\mathcal{R}(X,\pi, k)$ (resp. $\mathcal{R}(X,\pi, l)$) is infinite. 
\end{definition}
Let $M \subset X$ be a multisection of $\pi$ defined over $l/k$. Let $X_M$ be the relatively minimal elliptic surface obtained by resolving the singularities of $X\times_C M$. The projection onto $M$ induces an elliptic fibration $\pi_M: X_M \rightarrow M$ and we have an injection  $\alpha: \mathrm{MW }(X, \pi,l) \hookrightarrow \mathrm{MW }(X_M,\pi_M,l)$ yielding
\begin{equation}\label{eq: rank}
\mathrm{rank}\, \MW(X_M,\pi_M,l)\geq \mathrm{rank}\, \MW(X,\pi,l).
\end{equation}

Let $\iota:M\rightarrow X$ be the inclusion map and $\mathrm{id: } M \rightarrow M$, the identity map. Then $(\iota, \mathrm{id}):M \rightarrow X \times_C M$ induces a section $\sigma_M$ of $\pi_M$, namely the strict transform of $(\iota, \mathrm{id})$ under the birational map $X_M\rightarrow X \times_{C} M$. If the latter is linearly independent of the sections in the image of $\alpha$ then the inequality in (\ref{eq: rank}) is strict and, in particular, an application of Silverman's Specialization Theorem (\cite[Theorem C]{Silverman}) implies that  $\mathrm{rank}X_t(l)> \mathrm{rank}\, \MW(X,\pi,l)$ for all but finitely many $t$ in the image of the $l$-points of the map $\pi|_M: M \rightarrow C$. 

If, moreover, $M(l)$ is infinite, then $\mathcal{R}(X,\pi, l)$ is infinite.

There are several ways of checking that $\sigma_M : M \rightarrow X_M$ induces a section which is linearly independent from the image of $\MW(X,\pi,l)$. For explicit $(X, \pi, \MW(X,\pi,l))$ and $M$, one can compute the height matrix of $M$ and the image of $\MW(X,\pi,l)$ given the information of the intersection pairing for these divisors with each other and with the reducible components of the fibers of $\pi$. If this information is not available, alternative methods were given in \cite{SalgadoANT} and \cite{PastenSalgado}. We outline the latter for its relevance in the proof of Theorems \ref{thm: rank jump odd d} and \ref{thm: rank jump_even}.

The following statement is a slight simplification of \cite[Lemma 3.1]{PastenSalgado}, which we include here for the sake of self-containment.

\begin{lemma}\label{lem:indep_multisections}
Let $\pi:X \rightarrow C$ be an elliptic fibration and $M$ a saliently ramified multisection defined over a field $l$. Let $X_M$ and $\pi_M: X_M \rightarrow M$ be as above. Then the section $\sigma_M$ induced on $X_M$ by $(\iota, \mathrm{id}):M \rightarrow X \times_C M$ is linearly independent of the pull-back of the sections of $\pi$ in $\MW(X_M,\pi_M,l)$.
\end{lemma} 

\begin{proof}
We observe that one can weaken the hypothesis of \cite[Lemma 3.1]{PastenSalgado} and assume that the multisection $B$ admits at least one ramification point in $U_\pi$. The proof is then taken verbatim from \cite{PastenSalgado} by considering $T$ as the empty set. 
\end{proof}

\begin{corollary}\label{cor: rank jump}
Let $(X, \pi, C)$ and $M$ be as above. Assume that $M$ is defined over $l$ and $M(l)$ is infinite. Then \[
\mathcal{R}(X, \pi, l) \text{ is infinite.}
\]
In other words, $\pi$ has the potential rank jump property.
\end{corollary}

In the remainder of this article, we focus on elliptic K3 surfaces. We take advantage of the lattice theoretic techniques available in this setting to prove that in all, except possibly two 17-dimensional families of elliptic K3 surfaces with non trivial Mordell--Weil group, the generic K3 surface admits at least one elliptic fibration with a multisection that satisfies the hypothesis of Corollary \ref{cor: rank jump} over the field of definition of the elliptic fibration, i.e., that have genus 0 or 1 and in the latter case, positive rank over the given field.

\subsection{Lattice polarized K3 surfaces and genus 1 fibrations}

One of the main tools in the study of families of K3 surfaces with prescribed geometric properties is the translation of such properties into lattice-theoretic conditions on the Picard group of surfaces and then the application of the theory of lattice polarized K3 surfaces described in \cite{Do}. Here we give a short sketch of the main ideas in this theory.

The second cohomology group of any K3 surface is the unique (up to isometry) even unimodular lattice of signature $(3,19)$, often denoted by $\Lambda_{K3}$.

Let $\Gamma$ be an even lattice of $\rk(\Gamma)=\gamma\leq 19$ and signature $(1,\gamma-1)$. Let us assume that there exists a primitive embedding of $\Gamma$ in $\Lambda_{K3}$. Then, by the surjectivity of the period map of the K3 surfaces, there exists a K3 surface, $X_0$, defined over $\mathbb{C}$ such that $\Pic(X_0)\simeq \Gamma$.

\begin{definition}
    A $\Gamma$-polarized K3 surface is a pair $(X, j)$ where $X$ is a K3 surface and $j:\Gamma\hookrightarrow \Pic(X)$ is a primitive embedding of lattices.
\end{definition}

By considering the local deformations of the K3 surface $X_0$ such that $\Pic(X_0)=\Gamma$, one can construct a local moduli space $\mathcal{X}_\Gamma$ of isomorphism classes of $\Gamma$-polarized K3 surfaces, see \cite{Do}. Moreover, by gluing the local constructions and considering the restriction of the period map, one obtains a coarse moduli space of the $\Gamma$-polarized K3 surfaces. By this construction, one deduces that the family of the $\Gamma$-polarized K3 surfaces has dimension $20-\gamma$.

We observe that the Picard group of a very general K3 surface in the family of the $\Gamma$-polarized K3 surfaces is exactly $\Gamma$.

If a geometric property of a K3 surface can be completely described in terms of lattice polarized K3 surfaces, this, combined with the description of the moduli space of lattice polarized K3 surfaces, allows one to describe the family of K3 surfaces admitting the required geometric property.

The proposition below is an example of this phenomenon and is repeatedly used in what follows, hence even if it is well known, we recall here both the statement and the proof.

\begin{proposition}\label{prop: U-polarized}
    A K3 surface $X$ admits an elliptic fibration if, and only if, the lattice $U\simeq \left[\begin{array}{ll}0&1\\1&0\end{array}\right]$ is primitively embedded in $\Pic(X)$. In particular, the family of K3 surfaces admitting an elliptic fibration is 18-dimensional and coincides with the family of $U$-polarized  K3 surfaces.
\end{proposition}
\proof This result is well known, cf. \cite[Lemma 2.1]{Ko}. We briefly recall the idea of the proof. If a K3 surface admits an elliptic fibration, then there exists a smooth genus 1 curve, that is, a fiber of the fibration. We denote it by $F$ and observe that $F^2=0$ (by adjuction). Moreover, there exists a rational curve that is a section, denoted by $\mathcal{O}$. Then $\mathcal{O}^2=-2$ (by adjuction) and $F\mathcal{O}=1$ (since $\O$ is a section). Therefore, the lattice $U\simeq \langle F, F+\mathcal{O}\rangle$ is a sublattice of $\Pic(X)$. Since $U$ is a unimodular lattice, it does not have overlattices, so the embedding $\langle F,F+\mathcal{O}\rangle\subset \Pic(X)$ is primitive. The converse is more delicate; indeed, it is not true that each class with self-intersection 0 is necessarily the class of a smooth irreducible curve (and hence of a fiber of a genus 1 fibration). Nevertheless, given a primitive class in $\Pic(X)$ with self-intersection 0, it is true that up to reflections induced by smooth rational curves on the surface, one can assume that it is the class of a smooth irreducible curve, which necessarily has genus 1 by adjuction. So, let $U$ be primitively embedded in $\Pic(X)$. Considering (potentially) reflections in $-2$-curves, one can assume that one of the classes with self-intersection 0 is the class of a smooth irreducible curve of genus 1, say  $F$, and then $|F|$ defines a fibration onto $\mathbb{P}(H^0(X,\mathcal{O}_X(F))^{\vee})=\mathbb{P}^1$. Denoted $\{F, v\}$ the base of $U$, with $F$ as above, $v-F$ is a $-2$-class with positive intersection with $F$, i.e., it is an effective $-2$-class. Moreover, it intersects the fiber in one point and therefore it is either an irreducible curve, and in particular a section, or the union of several curves, among which there is a section. In both the cases the fibration admits a section.\endproof

    In the proof of the previous proposition we have that if there is a class of self-intersection 0 in the Picard group of a K3 surface $X$, then $X$ admits a genus 1 fibration. This is true independently on the presence of a copy of $U$ inside the Picard group, i.e., a K3 surface admits a genus 1 fibration if and only if its Picard group contains a class with self-intersection 0. Since the lattice spanned by a class with self-intersection 0 is degenerate, it cannot be used to construct a family of lattice polarized K3 surfaces. This implies that a projective K3 surface which admits a genus 1 fibration is polarized with a rank 2 lattice which represents 0, i.e. which contains at least one vector with self-intersection 0. There are infinitely many choices for these lattices, even up to isometries, which are classified and studied in \cite{vanGeemen}.

\subsubsection{K3 surfaces with more than one genus 1 fibration.}\label{subsubsec: K3 with two genus 1 fibration} If a K3 surface contains two distinct genus 1 fibrations, then the class of the fiber of the first one, say $F_1$, and the one of the second one, say $F_2$, have a positive intersection $F_1F_2=m>1$. The intersection has to be larger than 1 since the base of a genus 1 fibration on a K3 surface is rational and so $F_2$ cannot be a section. In particular, $\langle F_1,F_2\rangle\simeq U(m)$, with $m>1$, is primitively embedded in the Picard group of the surface. The generic K3 surface $X$ such that $\Pic(X)\simeq U(m)$ admits two genus 1 fibrations whose classes of fibers meet in $m$ points. 
Therefore the K3 surfaces that admit more than one genus 1 fibration are contained in the families of $U(m)$-polarized K3 surfaces for a certain $m>1$.

\subsubsection{K3 surfaces with infinite automorphisms group and one genus 1 fibration.} Not only the properties of the elliptic fibrations, but also the ones of the automorphism group of a K3 surface can be studied by considering its Picard lattice, and therefore by using the theory of lattice polarized K3 surfaces.  In a recent work, Brandhorst and Mezzedimi classified K3 surfaces that admit an infinite automorphim group and a unique genus 1 fibration (\cite[Theorem 5.2]{BM}). Gvirtz-Chen and Mezzedimi compiled a list of all Picard lattices of such K3 surfaces, and a list of the Picard lattices that contain $U$ as a sublattice, i.e., for which the genus 1 fibration admits a section, making it an elliptic fibration (see the proof of \cite[Theorem 4.5]{GM}). The list contains 49 lattices, among which 5 correspond to 17-dimensional families of elliptic K3 surfaces. The remaining 44 represent families of dimensions between 16 and 8 that arise as specializations of the 17-dimensional ones. The authors provide a magma code and a list of the 49 lattices on the website of the first author.

\subsubsection{K3 surfaces with an elliptic fibration and low Picard number}\label{subsubsec: K3 low Picard}
Let $(X,\pi,\mathbb{P}^1)$ be an elliptic fibration on the K3 surface $X$. Then there are no multiple fibers of $\pi$, due to the presence of a section, nevertheless there can be reducible fibers which contain some non-reduced components. These are the fibers of type $I_n^*$, $II^*$, $III^*$ or $IV^*$ in the Kodaira classification of the singular fibers, see \cite{Mi} for notation. Each of these fibers has at least 5 components, and in particular at least 4 components which do not intersect the zero section. These components are independent in $\Pic(X)$, and they are also independent with respect to the classes $F$ and $\mathcal{O}$ of the fiber and of the zero section of $\pi$, respectively. Hence, the presence of a fiber with non-reduced components forces the Picard number of the surface to be at least 6.

Let us now assume that $(X,\pi,\mathbb{P}^1)$ is an elliptic fibration on the K3 surface $X$ that admits a torsion section of order $n$. The translation by this section is a symplectic automorphism of order $n$. By \cite{Nik} there cannot be symplectic automorphism of finite order on a K3 surface $X$ if its  Picard number is less than 8.

The previous observations give the following well-known result, which will be useful in the following.
\begin{lemma}\label{lemma: e.f. with small rho}
Let $X$ be a K3 surface with $\rho(X)< 6$. Any elliptic fibration $\pi:X\ra\mathbb{P}^1$ admits neither fibers with non-reduced components nor torsion sections.
\end{lemma}

\section{Bisections of genus 1 and elliptic fibrations}\label{subsec:bisections lattice}

In what follows, we study bisections on elliptic K3 surfaces, with a focus on low-genus bisections and investigate, whenever possible, whether they are saliently ramified, due to their application to the rank jump problem (see \textit{Questions} 1 and 2). Since the property of being saliently ramified is reflected in the local intersection behavior of the multisection with the elliptic fibers (Remark \ref{rem:salient}), we provide geometric realizations of the surface and the fibration that allow for the identification of the desired behavior, i.e., non-reduced intersection with a smooth fiber.

Recall that K3 surfaces that admit an elliptic fibration $\pi:X\ra\mathbb{P}^1$ lie in the 18-dimensional space of $U$-polarized K3 surfaces (Proposition \ref{prop: U-polarized}). In a generic member of this family, each smooth multisection $M$ of $\pi$ of degree $m\geq 2$ has genus at least $m^2+1$. In fact, $M\in\Pic(X)=\langle F,\mathcal{O}\rangle$ can be written as $M=aF+b\O$. Then $MF=m$ implies that $b=m$, while $M\O=a-2m\geq 0$, being the number of points at the intersection of two irreducible curves, so $a\geq 2m$. By adjunction, $g(M)=(M^2+2)/2=am-m^2+1$, so $g(M)\geq m^2+1$. To find lower genus smooth multisections we must consider elliptic fibrations on K3 surfaces with Picard rank at least 3. Then we study all the infinitely many 17-dimensional subfamilies of K3 surfaces admitting an elliptic fibration: these are indexed by $d\in\mathbb{N}_{>0}$ and correspond to $U\oplus \langle -2d \rangle$-polarized K3 surfaces. We show that the generic members of all except three of these families admit a smooth low-genus bisection, that is, rational or of genus 1. The main results of this section are summarized in Theorem \ref{theorem:Picard3}.

\begin{notation}\label{notation Lambdad and Ld}
In what follows, we denote by $\Lambda_d:=U\oplus \langle -2d\rangle$ and $\mathcal{L}_d$ the family of the $\Lambda_d$-polarized K3 surfaces.
\end{notation}

\begin{theorem}\label{theorem:Picard3}
Let $X$ be a K3 surface with $\rho(X)=3$ and $X\in\mathcal{L}_d$. 

If $d=1$, then the unique elliptic fibration on $X$ has a reducible fiber, its Mordell--Weil group is trivial and there are no smooth rational or genus 1 bisections. 

For any other $d$, any elliptic fibration on $X$ has no reducible fibers, the Mordell--Weil group of any elliptic fibration on $X$ is isomorphic to $\Z$ and admits a generator that intersects the zero section in $d-2$ points. 

Moreover, the following hold:

\begin{itemize}
\item[a)] If $d\leq 3$, the unique elliptic fibration on $X$ does not admit smooth multisections of genus $g\leq 1$.

\item[b)] If $d>3$, then
\begin{itemize}
\item
any elliptic fibration on $X$ admits a smooth rational bisection if, and only if, $d\equiv 1\mod 2$ and this bisection is salient. 
\item any elliptic fibration on $X$ admits a smooth genus 1 bisection if, and only if, $d\equiv 0 \mod 2$ and this bisection is salient. 
\end{itemize}
\end{itemize}
\end{theorem}

We observe that by previous theorem, smooth rational and smooth genus 1 bisections of elliptic fibration cannot coexist, if $\rho(X)\leq 3$.

To prove this result we first consider the presence of specified bisections on K3 surfaces that admit a genus 1 fibration for which there are no sections, and then we extend the results to our context.

\subsection{Bisections of genus 1 fibrations}\label{subsect: genus 1 fib on K3}
Let $X$ be a K3 surface that admits a genus 1 fibration $\pi:X\ra \mathbb{P}^1$ without a section. The class of the fiber $F$ has trivial self-intersection, and there exists an ample class $H$ with $H^2=d>0$ among the generators of $\Pic(X)$, whose rank is forced to be at least  2. In other words, a general K3 surface $X$ with a genus 1 fibration and Picard number 2 has $\Pic(X)\simeq \Gamma_{b,c}$, where $\Gamma_{b,c}$ is the indefinite lattice of rank 2 defined as 
\begin{equation}\label{eq: def Gamma}
\Gamma_{b,c}:=\left\{\mathbb{Z}^2, (0,b,2c):= \begin{pmatrix}
0 & b \\
b & 2c 
\end{pmatrix}\right\},
\end{equation}
 where one can assume $0\leq c<b$, see \cite[Remark 4.2]{vanGeemen}.
Since $\pi$ does not admits a section, the lattice $\Gamma_{b,c}$ is not isometric to $U$. Since $U$ is the unique hyperbolic even unimodular lattice of rank 2, $\det(\Gamma_{b,c})\neq -1$ and so $b\neq 1$. 

The presence of a bisection implies that $b=2$. Following \cite{vanGeemen}, we observe that if $b=2$, there are two isomorphism classes of lattices $\Gamma_{2,c}$, namely $\Gamma_{2,0}$ and $\Gamma_{2,1}$. 

\begin{corollary}\label{cor:bisections}
Let $X$ be a K3 surface with a genus 1 fibration $\pi$ and $\rho(X)=2$. If $\pi$ admits a smooth bisection then it admits a smooth bisection of genus 0 or 1 and the two possibilities are mutually exclusive.
\end{corollary}

\begin{proof}
If $X$ admits a genus 1 fibration with a bisection and $\rho(X)=2$, then $\Pic(X)=\Gamma_{2,c}$ and by \cite[Proposition 3.7]{vanGeemen}, we can assume that $c=0$ or $c=1$. Let $D_1, D_2$ be two generators of $\Pic(X)$ with intersection matrix $\Gamma_{2,c}$. If $c=0$, then $D_1$ and $D_2$ are both classes of irreducible curves of genus 1 and $D_i$ is a bisection for the genus 1 fibration given by $D_j$, for $i,j\in \{1,2\}$ and $i\neq j$. If $c=1$, then one can assume that $D_1$ is the class of a curve of genus 1. Since $(D_2-D_1)^2=-2$ and $(D_2-D_1)D_1>0$, $(D_2-D_1)$ is an effective class, which corresponds to a smooth rational curve on the surface or to a reducible curve with at least one component which is a smooth rational curve. Since there are exactly two classes of self-intersection $-2$ in $\Gamma_{2,1}$ ($D_2-D_1$ and $D_1-D_2$) and only one of them is effective, we conclude that $(D_2-D_1)$ is a smooth rational curve. Since $(D_2-D_1)\cdot D_1=2$, it is a rational bisection of the genus 1 fibration induced by $D_1$.  

To show that the two possibilities are mutually exclusive, we prove that $\pi$ admits a rational (resp. genus 1) bisection if and only if $\Pic(X)\simeq \Gamma_{2,1}$ (resp. $\Pic(X)\simeq \Gamma_{2,0})$). We already prove the only if condition. So now assume that $\pi$ admits the required bisection $B$. Then $\Pic(X)$ contains the class $F$, of the fiber of $\pi$, and of $B$. Moreover, $F^2=0$ and $FB=2$ and $B^2=-2$ (resp. $B^2=0$). Since there are no non trivial overlattices of finte index of $\langle F,B\rangle$ which could appear as Picard group of a K3 surface, it follows that $\pi$ admits a rational bisection if and only if $\Pic(X)\simeq \langle F,B\rangle\simeq \left[\begin{array}{cc}0&2\\2&-2\end{array}\right]\simeq \Gamma_{2,1}$, where the isometry between the last two lattices was already proven (it suffices to compute the intersection form on $\langle F,2F+B\rangle$). Similarly, $\pi$ admits a genus 1 bisection if and only if $\Pic(X)\simeq \langle F,B\rangle\simeq \Gamma_{2,0}$. Since $\Gamma_{2,0}$ does not contains classes with self intersect 2 and $\Gamma_{2,1}$ does, they can not be isometric.
\end{proof}

In what follows, we give a geometric realization of K3 surfaces with a genus 1 fibration that admits a smooth bisection, according to the genus of the bisection. Moreover, we show that, for a general K3 in the family, the smooth bisection described in Corollary \ref{cor:bisections} is saliently ramified for the genus 1 fibration.

\begin{proposition}\label{prop: F rat smooth bisection}
The family $\mathcal{F}_{\rm{rat}}$ of K3 surfaces which admit a genus 1 fibration with a smooth rational bisection $B$ is 18-dimensional and coincides with the family of the $\langle 2\rangle\oplus \langle -2\rangle$-polarized K3 surfaces, i.e. with the family of the K3 surfaces which are double covers of $\mathbb{P}^2$ branched on a sextic with a simple node.

For generic $X$ in the family, the bisection $B$ is saliently ramified.

\end{proposition}
 
\proof
Keeping the notation introduced in the proof of Corollary \ref{cor:bisections}, the first part follows by considering $\{D_2, D_2-D_1\}$ as base of the Picard group.

One can assume that the divisor $D_2$ is nef and therefore $\varphi_{|D_2|}:X\ra \mathbb{P}^2$ is a $2:1$ cover branched on a sextic $C_6$. Since $D_2\cdot (D_2-D_1)=2-2=0$, $\varphi_{|D_2|}$ contracts the rational curve $B=D_2-D_1$ to a singular point $p$ of the branch sextic.
The class of a fiber $F$ is $D_1=D_2-B$, so it corresponds to the pencil $\mathcal{P}$ of lines in $\mathbb{P}^2$ passing through $p=\varphi_{|D_2|}(B)$. Each line $\ell\in \mathcal{P}$ intersects the branch sextic $C_6$ in $p$ (with multiplicity 2) and in other 4 points. Hence $\varphi_{|D_2|}$ restricted to $\varphi_{|D_2|}^{-1}(\ell)$ is $2:1$ cover of the rational curve  $\ell$ which is branched in 4 points; in particular it is a curve of genus 1, and indeed a fiber of $\pi=\varphi_{|D_1|}$.

The singular fibers of $\pi$ are mapped by $\varphi_{|D_2|}$ to lines of the pencil $\mathcal{P}$ for which one intersection point with $C_6-\{p\}$ has multiplicity higher than 1. 

To construct the smooth surface $X$ one first blows up $\mathbb{P}^2$ in $p$ and then considers the double cover branched on the strict transform of $C_6$. The exceptional divisor $E$ of the blow-up is not in the branch locus. It intersects each line of $\mathcal{P}$ in one point and hence its double cover $B$ intersects the fiber corresponding to a line in $\mathcal{P}$ in two points. In fact, $B$ is a bisection. The double cover $B\ra E$ branches in the two points corresponding to the lines in $\mathcal{P}$ that are the principal tangents of $C_6$ through the node $p$. Since the exceptional divisor $E$ parametrizes the lines in $\mathcal{P}$, we obtain that $\pi:X\ra \mathbb{P}^1\simeq E$. The previous considerations show the following.
\begin{itemize}
\item The singular fibers of $\pi$ are fibers over the points $q\in E\simeq \mathbb{P}^1$ such that $q=E\cap \widetilde{\ell_q}$ where $\widetilde{\ell_q}$ is the strict transform of a line $\ell_q\in\mathcal{P}$ and $\ell_q$ is a line which intersects $C_6-\{p\}$ in at least one point with multiplicity higher than 1;
\item The double cover $B\ra\mathbb{P}^1\simeq E$ branches on the points corresponding to the principal tangents of $C_6$ in $p$.
\end{itemize}
Generically, each principal tangent to $C_6$ in $p$ intersects $C_6$ in 3 other points with multiplicity 1, so generically it corresponds to a smooth fiber of the genus 1 fibration.
Hence, generically, $B$ branches on points corresponding to smooth fibers, and in particular it is saliently ramified. 
\endproof
\begin{remark}\label{rem: singular rational salient bisection case g=0}
In the setting of Proposition \ref{prop: F rat smooth bisection},
\begin{itemize}
    \item[a)]
 {\rm the result described on the genus 1 fibration induced by the pencil of the lines $\mathcal{P}$ remains true on certain  specializations of $C_6$. In particular, it remains true if the specialization $C_6'$ of $C_6$ is such that: the double cover of $\mathbb{P}^2$ branched on $C_6'$ is still a K3 surface $X$ and the point $p$ remains a simple node for $C_6'$; the latter condition  guarantees that the K3 surface $X$ admits a genus 1 fibration induced by $\mathcal{P}$ and that $B=\varphi_{|D_2|}^{-1}(p)$ is an irreducible smooth rational curve. One automatically obtains that $B$ is a bisection of the genus 1 fibration. So, even under these specializations, $B$ is a smooth rational bisection of a genus 1 fibration on a K3 surface and it is generically ramified on smooth fibers.}

 \item[b)] {\rm There could be special sextic curves $C_6$ for which a principal tangent in $p$ intersects $C_6$ in one point with multiplicity 1 and one point with multiplicity 2. If both principal tangents are of this type, then $B$ is a smooth rational bisection that is not saliently ramified.}

 \item[c)] {\rm  Let $\ell$ be a line not containing $p$ and bitangent to $C_6$, then $\varphi_{|D_2|}^{-1}(\ell)$ is a singular rational bisection of the fibration $\varphi_{|D_1|}$ and it is generically saliently ramified.}
\end{itemize}
\end{remark}

\begin{proposition}\label{prop: F g=1 smooth bisection}
The family $\mathcal{F}_{g=1}$ of K3 surfaces that admit a genus 1 fibration with a smooth genus 1 bisection $B$ is 18-dimensional and coincides with the family of $U(2)$-polarized K3 surfaces, i.e. with
the family of the K3 surfaces which are double covers of $\mathbb{P}^1\times \mathbb{P}^1$ branched on a curve of bidegree $(4,4)$.

The K3 surfaces in this family admit two elliptic fibrations with infinitely many smooth saliently ramified genus 1 bisections. 
\end{proposition}
\proof
Let $X$ be a K3 surface endowed with a genus 1 fibration $\pi:X\ra \mathbb{P}^1$ with a smooth genus 1 bisection. Denote by $F$ the class of the fiber and $B$ the class of the bisection, then the lattice $\langle F, B\rangle$ is isometric to $U(2)$ and the divisor $F+B$ defines a map $\varphi_{|F+B|}:X\ra\mathbb{P}^1\times \mathbb{P}^1\subset \mathbb{P}^3$ that realizes $X$ as a double cover of the quadric $\mathbb{P}^1\times \mathbb{P}^1$ branched on a curve $C_{4,4}$ of bidegree $(4,4)$.
The equation of $C_{4,4}$ is of the following type 
$$x_0^4a_4(y_0:y_1)+x_0^3x_1b_4(y_0:y_1)+x_0^2x_1^2c_4(y_0:y_1)+x_0x_1^3d_4(y_0:y_1)+x_1^4e_4(y_0:y_1),$$ with $a_4,b_4,c_4,d_4,e_4$ homogeneous polynomials of degree 4.
The singular fibers correspond to the inverse image on $X$ of the curves $D_{(\alpha,\beta)}=V(\alpha x_0-\beta x_1)$ of bidegree $(1,0)$,  such that $D_{(\alpha,\beta)}\cap C_{4,4}$ contains at least one point with multiplicity greater than 1. There is a finite number of such curves, generically 24, which correspond to the 24 singular fibers of type $I_1$ (that is, nodal) of a genus 1 fibration on $X$. Let us now consider a $(0,1)$-curve $\ell_{(\gamma,\delta)}$ with equation $\gamma y_0-\delta y_1=0$. For a generic choice of $\gamma$ and $\delta$, $\ell_{(\gamma,\delta)}\cap C_{4,4}$ consists of 4 distinct points, so that the inverse image of $\ell_{(\gamma,\delta)}$ is a double cover of a rational curve branched in 4 points, i.e. it is a smooth genus 1 curve on $X$. Since $\ell_{(\gamma,\delta)}$ is a section of the fibration $\mathbb{P}^1_x\times \mathbb{P}^1_y\ra\mathbb{P}^1_x$, its inverse image in the double cover is a bisection of $\pi:X\ra\mathbb{P}^1_x$.  
Again for a generic choice of  $\gamma$ and $\delta$, the points $\ell_{(\gamma,\delta)}\cap C_{4,4}$ are not contained in any curve $D_{(\alpha, \beta)}$ corresponding to a singular fiber. This guarantees that the bisection which is the double cover of $\ell_{(\gamma,\delta)}$ is branched on smooth fibers, i.e., is saliently ramified. So we find infinitely many saliently ramified genus 1 bisections of the genus 1 fibration on $X\ra\mathbb{P}^1_x$.\endproof

\begin{remark}\label{rem: singular rational salient bisection case g=1}{\rm Parallel to Remark \ref{rem: singular rational salient bisection case g=0}(c), in the setting of the proof of Proposition \ref{prop: F g=1 smooth bisection}, it is possible to specialize $\ell_{(\gamma,\delta)}$ in such a way that two of the intersection points with $C_{4,4}$ coincide. In this case, its inverse image is a singular rational curve, which is still a bisection of $\pi$. For generic choice of $C_{4,4}$, this singular rational bisection is saliently ramified (with the same argument we used in the general context).

This specialization depends on the choice of $\ell_{(\gamma,\delta)}$, and does not require a specialization of $C_{4,4}$, hence  it appears on the general member of the family of K3 surfaces we are considering.}
\end{remark}

\subsection{Elliptic K3 surfaces of Picard number 3}\label{subsec:Picard3}

We specialize further in the families discussed in Subsection \ref{subsect: genus 1 fib on K3} and consider K3 surfaces with an elliptic fibration that admits a bisection. The coexistence of a section and of a bisection of low genus, implies that the Picard number of such K3 surfaces is at least 3, since both $U$ and $\Gamma_{2,c}$ have to be primitively embedded in the Picard lattice (see Propositions \ref{prop: U-polarized} and \ref{prop: F rat smooth bisection}). We deal with the general case, i.e., Picard number 3. We first present general results (Lemma \ref{lemma:familyLd}) to later restrict to bisections.

\begin{lemma}\label{lemma:familyLd}
Let $X$ be an elliptic K3 surface with Picard number  3.
Then there exists $d>0$ such that $\Pic(X)\simeq U\oplus\langle -2d\rangle$. Moreover, for every elliptic fibration $\pi:X\ra\mathbb{P}^1$, \begin{itemize}
\item  $\rk(\MW(\pi))=0$ if and only if $d=1$ if and only if $\pi$ admits a reducible fiber;
\item $\rk(\MW(\pi))=1$ if and only if $d>1$ if and only if $\pi$ has no reducible fiber. In this case $\MW(\pi)$ is generated by a section $S_1$ that intersects the zero section in $d-2$ points.\end{itemize}
\end{lemma}
\proof The result was already presented in \cite[Sec. 4.3]{G}, but for the reader's convenience we recap the key ideas. Since $X$ admits an elliptic fibration, $U$ has to be primitively embedded in its Picard group, see Proposition \ref{prop: U-polarized}. Since the Picard number is 3, and the Picard lattice has signature $(1,2)$, the orthogonal complement of $U$ in the Picard group is $\langle -2d\rangle$ for a certain positive $d$, and $\Pic(X)$ is an overlattice of finite index of $U\oplus \langle -2d\rangle$. But there are no even lattices which are overlattices of finite non-trivial index of $U\oplus \langle -2d\rangle$ in which $\langle -2d\rangle$ is primitive, so $\Pic(X)\simeq U\oplus\langle -2d\rangle$. This argument holds for every elliptic fibration on $X$, i.e. every elliptic fibration on $X$ is associated to an embedding of $U$ in $\Pic(X)$, whose orthogonal complement is $\langle -2d\rangle$. So all the geometric properties that depend only on the intersection numbers of classes in $\Pic(X)$ are common to all elliptic fibrations on $X$.
We denote by $u_1,u_2,b_3$ the base of $\Pic(X)$ on which the bilinear form is $U\oplus\langle -2d\rangle$,  and we fix $F:=u_1$ and $\mathcal{O}=u_2-u_1$. By Lemma \ref{lemma: e.f. with small rho}, the torsion part of $\MW(\pi)$ is trivial.

If $d=1$, then the vector $b_3$ is a $(-2)$-class and hence (up to a sign) the class of an irreducible component of a reducible fiber that does not intersect the zero section. By the Shioda--Tate formula, the presence of this reducible fiber is equivalent to $\MW(\pi)=\{1\}$. 

If $d>1$, then there are no $(-2)$-classes orthogonal to $U$ and so $\pi$ has no reducible fibers. Again by the Shioda--Tate formula, this is equivalent to $\MW(\pi)\simeq \Z$. 
The classes of the sections of $\pi$ are 
$$S_n:=dn^2F+\mathcal{O}+nb_3$$ and the isomorphism $\MW(\pi)\simeq \Z$ maps $\mathcal{O}$ to $0$ and $S_n$ to $n$. In particular, if $\Pic(X)\simeq U\oplus \langle -2d\rangle$ and $d>1$, then $X$ admits an elliptic fibration without reducible fibers such that the generator $S_1$ of $\MW(\pi)$ intersects the zero section in $d-2$, i.e. $S_1\mathcal{O}=d-2$. Observe that $S_1\O=S_{-1}\O$ and more in general $S_n\O=S_{-n}\O$.\endproof

Recalling Notation \ref{notation Lambdad and Ld}, by the previous discussion, if $d>1$ the 17-dimensional family $\mathcal{L}_d$ is the family whose general members are K3 surfaces that admit an elliptic fibration $\pi$ without reducible fibers with $\rk(\MW(\pi))=1$  such that the generator of $\MW(\pi)$ intersects the zero section in $d-2$ points. 

\begin{remark}\label{rmk:one type of elliptic fibration}
As explained in the proof of Lemma \ref{lemma:familyLd}, there could be more than one elliptic fibration on a K3 surface $X\in\mathcal{L}_d$ for a determined values of $d$. Nevertheless, if $X$ is general in its family, in particular if $\rho(X)=3$, then all these elliptic fibrations admit the same configurations of reducible fibers, their Mordell--Weil groups are isomorphic and the generator of the Mordell--Weil group intersects the zero section in the same number of points (briefly, their Mordell--Weil lattices are isometric). 
\end{remark}

As already recalled in Section \ref{section: preliminaries}, there is a classification of the K3 surfaces that admit a unique elliptic fibration given by Brandhorst and Mezzedimi (\cite{BM}) and Mezzedimi and Gvirtz-Chen (\cite{GM}).

\begin{proposition}\label{prop: unique elliptic fibration}
Let $X$ be a K3 surface with $\rho(X)=3$, with a unique genus 1 fibration. Assume moreover that the genus 1 fibration admits at least one section, i.e., it is an elliptic fibration. Then $X$ belongs to one of the six families $\mathcal{L}_d$ for $d=1,2,3,5,7,13$. Moreover, $\mathrm{Aut}(X)$ is infinite if, and only if, $d>1$.
\end{proposition}
\begin{proof}
See the proof of \cite[Theorem 4.5]{GM}. The authors provide a magma code and a list of the 49 lattices associated with K3 surfaces $X$ with a unique genus 1 fibration in the website of the first author. The 49 lattices arise as specializations of at least one of the 6 lattices of rank 3 listed in the statement.
\end{proof}

If $X\in\mathcal{L}_d$ and $d>1$, then each elliptic fibration admits an infinite order section which yields, by translation, an automorphism of infinite order preserving the elliptic fibration. We discuss, in what follows, some properties of the automorphisms group of $X\in\mathcal{L}_d$.

\subsubsection{Automorphisms of K3 surfaces in $\mathcal{L}_d$}	

Let $X$ be a general K3 surface in a family $\mathcal{L}_d$. Fix the base $\{u_1,u_2,b_3\}$ of $\Lambda_{d}$ considered in proof of Lemma \ref{lemma:familyLd}.
The elliptic involution on each fiber of $\pi$ extends to an automorphism on $X$ that acts on the chosen base of $\Lambda_d$ as follows:
$$\varepsilon=\left[\begin{array}{rrr}1&0&0\\
0&1&0\\
0&0&-1\end{array}\right]$$

If $d\geq 2$, the translation by the section $S_1$ is an automorphism of infinite order of the K3 surface, which acts on the base $\{F,\mathcal{O},b_3\}$ as follows:
$$T_1=\left[\begin{array}{rrr}1&d&2d\\
0&1&0\\
0&1&1
\end{array}\right]$$
(indeed, by definition, $T_1(F)=F$, $T_1(\mathcal{O})=S_1$ and $T_1(S_1)=S_2$).
	
If $d=1$, the previous matrix is still associated with an isometry of the lattice $\Lambda_{d}$, but this isometry cannot be associated with an automorphism since it maps $F-b_3$ to $-F-b_3$. The former is an effective class corresponding to the component of  $I_2$-fiber meeting the zero section. The latter is not effective since $(-F-b_3)\mathcal{O}=-1$.

If $d>1$, the group $\langle T_1, \varepsilon\rangle$ is also generated by $\varepsilon$ and $ T_1\circ\varepsilon$, and hence it is isomorphic to $\left(\mathbb{Z}/2\mathbb{Z}\right)*\left( \mathbb{Z}/2\mathbb{Z}\right)$.

    If $d\neq 1,2,3,5,7,13$, the surface $X$ admits more than one elliptic fibration. Let $\pi'$ be an elliptic fibration, then associated with $\pi'$ there are two automorphisms $\varepsilon'$ and $T_1'$ that are respectively induced by the elliptic involution and by the translation on each fiber. If $\pi$ and $\pi'$ are not equivalent up to automorphisms, $\varepsilon'$ and $T_1'$ are not conjugate to $\varepsilon$ and $T_1$ in $\mathrm{Aut}(X)$.

\begin{remark} If $X$ is general in its family, then it is $\mathrm{Aut}$-general, i.e., its automorphism group contains only symplectic automorphisms of infinite order (e.g. $T_1$) and non-symplectic automorphisms of order 2 (e.g. $\varepsilon$). In particular, this implies that $X$ does not admit isotrivial elliptic fibrations with $j$-invariant of the generic fiber equal to $0$ or $1728$. Indeed, otherwise the complex multiplication on each fiber would induce an automorphism on the surface which is non--symplectic and has order 3 or 4, respectively. If an elliptic fibration is isotrivial with a constant $j$-invariant not equal to 0 or 1728, then the reducible fibers are forced to be of type $I_0^*$, and in particular with non reduced components. By Lemma \ref{lemma: e.f. with small rho}, these isotrivial fibrations cannot appear if the Picard number of $X$ is at most 3, i.e. for the surfaces considered in this paper. \end{remark}

In what follows, we continue the discussion on elliptic K3 surfaces with Picard number 3 that admit a smooth bisection. We present our results according to the genus of the smooth bisection.

\subsection{Smooth rational bisections on elliptic K3 surfaces}\label{subsec: Frat x}

Let $X$ be a K3 surface with an elliptic fibration $\pi:X\ra \mathbb{P}^1$ and a smooth rational bisection $B$. Then, in particular $X$ belongs to a subfamily of $\mathcal{F}_{\rm{rat}}$ considered in Proposition \ref{prop: F rat smooth bisection} and, of course, to the family of K3 surfaces admitting an elliptic fibration. 
	
	\begin{definition} Let $\mathcal{F}_{\rm{rat}}^x$ be the family of K3 surfaces with an elliptic fibration admitting a smooth rational bisection which intersects the zero section in $x\geq 0$ points (possibly counted with multiplicity).\end{definition}

\begin{proposition}\label{prop:polarization rational}
The family $\mathcal{F}_{\rm{rat}}^x$ and the family $\mathcal{L}_{2x+5}$ coincide.
\end{proposition}	
\proof Let $X$ be a general K3 surface in $\mathcal{F}_{\rm{rat}}^x$ and recall that $x\geq 0$ since it is the intersection number of two irreducible distinct curves. The classes $F$, of the fiber of the fibration, $\mathcal{O}$, of the zero section of the fibration, and $B$, of the rational bisection, are contained in the Picard lattice of $X$. Let us denote $\Phi_x:=\left[\begin{array}{rrr}0&1&2\\1&-2&x\\2&x&-2\end{array}\right]$ the lattice $\langle F,\mathcal{O},B\rangle$. 
By the base change $\{F,\mathcal{O}, (-4-x)F-2\mathcal{O}+B\}$, we obtain a lattice isometric to $\Phi_x$ which is necessarily primitively embedded $\Pic(X)$. Indeed, we already observed that the sublattice $\langle F,\mathcal{O}\rangle$ has to be primitively embedded in $\Pic(X)$; the class $(-4-x)F-2\mathcal{O}+B$ is not a divisible class in $\Pic(X)$, since the coefficient of $B$ is 1; there are no even overlattices of finite index of $\langle F,\mathcal{O}, (-4-x)F-2\mathcal{O}+B\rangle$ in which both $\langle F,\mathcal{O}\rangle$ and $\langle (-4-x)F-2\mathcal{O}+B\rangle$ are primitively embedded. Hence, we proved that the family $\mathcal{F}_{\rm{rat}}^x$ is the family of the $\Phi_x$-polarized K3 surfaces. 

To compare the families $\mathcal{F}_{\rm{rat}}^x$ and $\mathcal{L}_{2x+5}$, it suffices to show that the lattices $\Phi_x$ and $\Lambda_{2x+5}$ are isometric.  
We consider again the base $\{F, F+\mathcal{O}, (-4-x)F-2\mathcal{O}+B\}$ of $\Phi_x$. The intersection form on this base is $U\oplus \langle -2(2x+5)\rangle$. This proves that $\Pic(X)$ is isometric to $\Lambda_{d}$ if $d=2x+5$, with $x\geq 0$.\endproof

In the proof of Proposition \ref{prop:polarization rational}, we give the base change from $\{F,\mathcal{O}, B\}$ to $\{u_1,u_2,b_3\}$.  The inverse base change gives $B=b_3+(4+x)F+2\mathcal{O}=(2+x)u_1+2u_2+b_3$. So if $X\in\mathcal{L}_d$ with $d\equiv 1\mod 2$ and $d\geq 5$, where $2x+5=d$, the class of the smooth rational bisection for $\pi$ is 
\begin{equation}\label{eq: explcit of bisection}B=\frac{d+3}{2}F+2\mathcal{O}+b_3.\end{equation}
\begin{remark}\label{rem; intersection B and sections}
In the previous situation (i.e. $\NS(X)\simeq \Lambda_d$, $d\equiv 1\mod 2$, $d\geq 5$), we have shown that every elliptic fibration on $X$ admits a smooth rational bisection $B$ such that $B\mathcal{O}=x=\frac{d-5}{2}$, whose class is given in \eqref{eq: explcit of bisection}. The parity of $B\mathcal{O}$ is the same as the one of $BS_n$ for every $n\in\mathbb{Z}$ and depends on $d$. More specifically, $BS_n=\frac{d+3}{2}+2dn(n-1)-4\equiv \frac{d+3}{2}\mod 4$, i.e. $BS_n\equiv x\mod 4$. In particular:  
\begin{eqnarray}\label{eq: intersection B sections}\begin{array}{c}B\mathcal{O}\equiv BS_n\equiv 0\mod 2 \Leftrightarrow d\equiv 1\mod 4\\
B\mathcal{O}\equiv BS_n\equiv 1\mod 2 \Leftrightarrow d\equiv 3\mod 4.\end{array}\end{eqnarray}
\end{remark}

\begin{remark}\label{rem: the case d=3, spltting of the bisection}
If $\NS(X)\simeq \Lambda_3$ (i.e. $d=3$), then the class $B=3F+2\mathcal{O}+b_3$ as in \eqref{eq: explcit of bisection} is still a $(-2)$-class with intersection 2 with $F$, but it does not correspond to an irreducible curve, since $B\mathcal{O}=-1<0$ and $\mathcal{O}$ is an irreducible curve. Indeed, the class $B$ is not an irreducible bisection, but the sum of two sections, namely $S_1=3F+\mathcal{O}+b_3$ and $\mathcal{O}$.\end{remark}

The automorphisms $T_1$ and $\varepsilon$ act as follows on the bisection $B$:
$$T_1(B)=6(x+2)F-4\mathcal{O}+3B,\ \ \varepsilon(B)=2(4+x)F+4\mathcal{O}-B \mbox{ and } T_1(\varepsilon(B))=B.$$ In particular, there is an involution of $X$ that fixes the multisection $B$. We observe that $T_1^n(B)$ and $\varepsilon(B)$ are not equal to $B$, and since they are isometries and $F$ is fixed by both of them, the curves $T_1^n(B)$ and $\varepsilon(B)$ are other rational bisections of the fibration.

By the proof of Proposition \ref{prop: F rat smooth bisection},  K3 surfaces in $\mathcal{F}_{rat}^x$ admit natural geometric realizations as a double cover of the plane branched on an irreducible nodal sextic. The presence of a section for the genus 1 fibration is equivalent to the splitting of a certain plane curve on the double cover. We discuss this in detail in what follows.

\begin{proposition}\label{prop: Fratx cover of P2}
Let $X\in\mathcal{L}_{2x+5}$, then there exists $f:X\ra \mathbb{P}^2$ a double cover of $\mathbb{P}^2$ branched on a plane sextic $C_6$ with a simple node at a point $p$, such that there exists a rational curve $R\subset\mathbb{P}^2$ of degree $x+1$ passing through $p$ with multiplicity $x$ that splits in the double cover.

The pencil of lines through $p$ induces an elliptic fibration on $X$ whose zero section is one of the components of $f^{-1}(R)$.

The cover involution of $f:X\ra\mathbb{P}^2$ is the involution $ T_1\circ \varepsilon$ and preserves the class of the bisection $B$.
\end{proposition}
\proof
For each $x\geq 0$, $\mathcal{F}_{\rm{rat}}\supset \mathcal{F}_{\rm{rat}}^x$ and indeed $\langle F,F+B\rangle$ is a copy of the lattice $\Gamma_{2,1}$ primitively embedded in $\Phi_x\simeq \Lambda_{2x+5}$.
By Proposition \ref{prop: F rat smooth bisection},  the surfaces in the 18-dimensional family $\mathcal{F}_{\rm{rat}}$ are realized as double covers of $\mathbb{P}^2$ branched on a sextic $C_6$ with a node at a point $p$, with a genus 1 fibration given by the pull-back of the pencil of lines through $p$; this  realization holds for $X\in\mathcal{F}_{\rm{rat}}^x$. Indeed, the linear system of $H=F+B$ gives the $2:1$ map $\varphi_{|H|}:X\ra \mathbb{P}^2$. As discussed in Proposition \ref{prop: F rat smooth bisection}, $B$ is contracted to the singular point $p$ of $C_6$ by the map $\varphi_{|H|}$. 

The class of the zero section $\mathcal{O}$ is such that $\mathcal{O}H=1+x$ and $\mathcal{O}B=x$.
Since $\mathcal{O}$ intersects each fiber of $\pi$ in exactly one point, its image  $R:=\varphi_{|H|}(\mathcal{O})\subset \mathbb{P}^2$ intersects each line at exactly one point different from $p$ and $R$ splits in the double cover, i.e. $(\varphi_{|H|})^{-1}(R)=\mathcal{O}\cup D$ where $D$ is another section of the fibration $\pi$. Since $\mathcal{O}H=x+1$, $\mathcal{O}B=x$ and $R$ splits in the double cover, $R$ is a plane curve of degree $x+1$ which passes through $p$ with multiplicity $x$. 

We observe that $\mathcal{O}\cap D$ are the inverse image of the points $R\cap C_6$ different from $p$ (since the point $p$ was blown up). In particular, $\mathcal{O}D=(6(x+1)-2x)/2=2x+3$ where we used that: the degree of $C_6$ is 6; the one of $R$ is $(x+1)$; the multiplicity of $C_6$ in $p$ is 2; the one of $R$ is $x$; each intersection point between $C_6$ and $R$ has multiplicity two when considered on $X$, so the number of intersection points on $X$ is half of the intersection product considered in $\mathbb{P}^2$.
Hence $\mathcal{O}D=2x+3=(2x+5)-2=d-2$, which is the intersection number between $\mathcal{O}$ and either $S_1$ or $S_{-1}$. 

The cover involution of $\varphi_{|H|}:X\ra \mathbb{P}^2$ preserves the classes $H$ and $B$. Recalling that $H=F+B$, we obtain that the cover involution preserves the classes $F$ and $B$, and maps $\mathcal{O}$ to $D$. In particular, it coincides with the involution $ T_1\circ\varepsilon$, which maps $\mathcal{O}$ to $S_{1}$, which implies that $D=S_{1}$.
\endproof

Since $S_{n}H=2n^2d- 2nd+\frac{d-3}{2}$ and $S_{n}B=2n^2d- 2nd+\frac{d-5}{2}$, the sections $S_{n}$ of the fibration $\varphi_{|F|}$ are mapped by $\varphi_{|H|}$ to plane curves of degree $2n^2d- 2nd+\frac{d-3}{2}$ passing through $p$ with multiplicity $2n^2d- 2nd+\frac{d-5}{2}$ and splitting in the double cover $\varphi_{|H|}:X\ra \mathbb{P}^2$. 
\begin{remark}\label{rem: d=3 double cover of the plane} In case $\NS(X)\simeq \Lambda_3$, we already observed that there is a class $B=3F+2\mathcal{O}+b_3$ such that $B^2=-2$ and $BF=2$, but it does not correspond to an irreducible curve. Similarly, one observes that the class $H=F+B=4F+2\mathcal{O}+b_3$ is still associated to a $2:1$ map $\varphi_{|H|}:X\ra\mathbb{P}^2$, but in this case the branch locus of the double cover has a cusp and not a simple node, indeed $H\mathcal{O}=HS_1=0$ and hence $\varphi_{|H|}$ contracts two smooth irreducible curves which form an $A_2$-configuration to a point (see also \cite[Section 6.3]{FM}).\end{remark}
\subsection{Smooth genus 1 bisections on elliptic K3 surfaces}

Let $X$ be a K3 surface which admits an elliptic fibration $\pi:X\ra \mathbb{P}^1$ and a genus 1 bisection $B$.  

\begin{definition} Let $\mathcal{F}_{g=1}^x$ be the family of elliptic K3 surfaces admitting a smooth genus 1 bisection which intersects the zero section in $x$ points (possibly counted with multiplicity).\end{definition}
\begin{lemma}\label{lemma: family Fg1x} The family $\mathcal{F}^x_{g=1}$ coincides with the family $\mathcal{L}_{2x+4}$. 
\end{lemma}	
\proof The proof is analogous to the one of the case of rational bisection (see Proposition \ref{prop:polarization rational}), by observing that the classes $F,\ \mathcal{O},\ B$ have the following intersection properties $F^2=B^2=0$, $F\mathcal{O}=1$, $FB=2$ and $B\mathcal{O}=x$ and that there is a base change which gives $u_1=F$, $u_2=F+\mathcal{O}$ and $b_3=-(x+4)F-2O+B$.\endproof

Hence the K3 surfaces which admit an elliptic fibration with a smooth genus 1 bisection which intersects the zero section in $x$ points, correspond to the K3 surfaces that are $\Lambda_{d}$ polarized with $d\geq 4$ and $d\equiv 0\mod 2$.

The base change between $\{F,\mathcal{O}, B\}$ and $\{u_1,u_2,b_3\}$ gives $B=b_3+(4+x)F+2\mathcal{O}=(2+x)u_1+2u_2+b_3$, and this allows one to compute the image of $B$ for the automorphisms $T_1$, $\varepsilon$ and their composition:
$$T_1(B)=6(x+2)F-4\mathcal{O}+3B,\ \ \varepsilon(B)=2(4+x)F+4\mathcal{O}-B \mbox{ and } T_1(\varepsilon(B))=B.$$ In particular we observe that there is an involution of $X$ which fixes the multisection $B$, and that the curves $T_1^n(B)$ and $\varepsilon(B)$ are other genus 1 bisections of the fibration.

\begin{proposition}\label{prop: Fratx cover of P1xP1}
Let $X\in\mathcal{L}_{2x+4}$, then there exist $f:X\ra \mathbb{P}^1\times\mathbb{P}^1\subset \mathbb{P}^3$, a double cover of $\mathbb{P}^1 \times \mathbb{P}^1$ branched over a smooth curve $C_{4,4}$ of bidegree $(4,4)$, and a rational curve $R\subset\mathbb{P}^1 \times \mathbb{P}^1$ of bidegree $(1,x)$ that splits in the double cover.

The pencil of $(0,1)$-curves induces an elliptic fibration on $X$ whose zero section is one of the components of $f^{-1}(R)$.

The cover involution of $f:X\ra\mathbb{P}^1 \times \mathbb{P}^1$ is the involution $ T_1 \circ\varepsilon$. It preserves the class of the bisection $B$.
\end{proposition}
\proof  The proof is analogous to the one in the case of rational bisection (see Proposition \ref{prop: Fratx cover of P2}).
\begin{remark} We observe that if $\NS(X)\simeq \Lambda_2$, the class $B=3F+2\mathcal{O}+b_3$ is still a class such that $B^2=0$ and $BF=2$, but it does not correspond to an irreducible curve of genus $1$ since $B\mathcal{O}=-1$. Therefore, it is not a bisection, but the sum of two sections and a fiber, indeed $B=S_1+\mathcal{O}+F$. A geometric description of this surface is given in \cite{G}, where it is shown that the map $\varphi_{|4F+2\mathcal{O}|}$ exhibits $X$ as a double cover of a cone over the rational normal curve in $\mathbb{P}^4$. A Weierstrass equation of the elliptic fibration on this surface is given in \cite[Section 6.2]{FM}.\end{remark}
\begin{remark}
Similarly to the Remark \ref{rem; intersection B and sections}, we observe that $B\mathcal{O}\equiv BS_n\equiv 1\mod 2$ if $d\equiv 2\mod 4$ and $B\mathcal{O}\equiv BS_n\equiv 0\mod 2$ if $d\equiv 0\mod 4$.\end{remark}
\subsection{Proof of Theorem \ref{theorem:Picard3}}

We recap the previous results to give an explicit proof of Theorem \ref{theorem:Picard3}.

In Lemma \ref{lemma:familyLd}, we proved that the Picard lattice of a K3 surface with Picard number 3 admitting an elliptic fibration is isometric to $\Lambda_d$ for a positive integer $d$ and we described the properties of the elliptic fibrations on a K3 surface that is $\Lambda_d$-polarized.

Then we defined the families $\mathcal{F}_{\rm{rat}}^x$ (resp. $\mathcal{F}_{g=1}^x$) as the families of K3 surfaces admitting an elliptic fibration with a smooth rational (resp. genus 1) bisection which intersects the zero section in $x$ points. In particular, the set of the K3 surfaces admitting an elliptic fibration with a smooth rational (resp. genus 1) bisection is $\bigcup_{x\in\mathbb{N}}\mathcal{F}_{\rm{rat}}^x$ (resp. $\bigcup_{x\in\mathbb{N}}\mathcal{F}_{\rm{g=1}}^x$). 

By Proposition \ref{prop:polarization rational}, $\mathcal{F}_{\rm{rat}}^x$ coincides with the family $\mathcal{L}_{2x+5}$ of the $U\oplus \langle -2(2x+5)\rangle$-polarized K3 surfaces. Hence, a K3 surface $X$ with Picard number 3 admits a smooth rational bisection if and only if there exists $x\in\mathbb{N}$ such that $X$ is a general  member of the family $\mathcal{F}_{\rm{rat}}^x\simeq \mathcal{L}_{2x+5}$, i.e. if and only if its Picard lattice is $U\oplus \langle -2d\rangle$ for $d=2x+5$ and $x\geq 0$. This implies that $d\equiv 1\mod 2$ and $d\geq 5$.

Similarly, by Lemma \ref{lemma: family Fg1x}, $\mathcal{F}_{\rm{g=1}}^x$ coincides with the family $\mathcal{L}_{2x+4}$ of the $U\oplus \langle -2(2x+4)\rangle$-polarized K3 surfaces. Hence, a K3 surface $X$ with Picard number 3 admits a smooth genus 1 bisection if and only if its Picard lattice is $U\oplus \langle -2d\rangle$ for $d=2x+4$ and $x\geq 0$. This implies that $d\equiv 0\mod 2$ and $d\geq 4$.

\section{Higher degree multisections}\label{section: higher degree}
We consider here multisections of degree higher than 2. The main ideas and the proofs are very similar to those in the previous section, where we considered multisections of degree 2. Indeed, as in that case, we discuss the 18-dimensional (resp. 17-dimensional) families of K3 surfaces with a genus 1 fibration (resp. elliptic fibration) admitting a smooth rational or genus 1 multisection of degree $m>2$. Moreover, for low values of $m$, namely $m=3,4,5$, one is also able to describe projective models of the generic K3 surface in the constructed families. %The main difference with respect to the case of degree $m=2$ is that now there could exist genus 1 fibrations admitting smooth multisections of degree $m>2$ but none of them are rational or of genus 1.
\begin{lemma}
Let $X$ be a K3 surface with a genus 1 fibration $\pi$ and $\rho(X)=2$. Then $\Pic(X)\simeq \Gamma_{m,c}:=\left[\begin{array}{cc}0&m\\m&2c\end{array}\right]$ where $m$ is the minimal degree of a multisection of $\pi$ and $-1\leq c\leq m-2$. Moreover, the following hold.
\begin{itemize}
\item[i)] The degree of any multisection of $\pi$ is a multiple of $m$. 
\item[ii)] There is a smooth rational multisection of degree $d$ of $\pi$ if and only if $d=m$, where $m$ is the minimal degree of a multisection of $\pi$, and $\Pic(X)\simeq \Gamma_{m,-1}$.
\item[iii)]If $d=gcd(m,c)$ and $m=dm'$, then there there is a smooth genus 1 multisection of degree $mm'$.
\end{itemize}
\end{lemma}
\begin{proof} The class of the fiber of $\pi$, $F$, is primitive in $\Pic(X)$, and hence there exists a basis of $\Pic(X)$ which is given by $\{F,M\}$ with $FM=m>0$. In particular, chosen $M$ to be the class of an irreducible curve, $M$ is a multisection of degree $m$ of $\pi$. The intersection form computed on these classes is $\Gamma_{m,c}$ and the condition on $c$ up to isometries of the lattices is given in \cite[Proposition 3.7]{vanGeemen}. Each class in $\Pic(X)$ can be written as $D=\alpha F+\beta M$, $\alpha,\beta\in \Z$. Therefore, $FD=m\beta$, which yields i). If $D$ is an irreducible curve, and in particular a multisection, then $FD=m\beta>0$ and it is $m$ if and only if $\beta=1$. The self-intersection of $D$ is $2\beta(\alpha m+c\beta)$. If $D$ is the class of a smooth rational multisection, then $\beta(\alpha m+c\beta)=-1$, which implies $\beta=1$ and $c=-1-\alpha m$. In particular, $D=\alpha F+M$, which is not irreducible unless $\alpha=0$. The irreducible rational multisection is then obtained for $\alpha=0$ and so $c=-1$. So $D$ coincides with $M$ and $D$ has degree $FM=m$. Viceversa, if $\Pic(X)\simeq \Gamma_{m,-1}$, then $M$ (the second generator of the basis of $\Pic(X)$) is a rational multisection. This concludes the proof of ii).

Finally, for iii), if $D$ is the class of a smooth genus 1 multisection, then $\alpha m+c\beta=0$. In particular, if $c=dc'$, the class $-c'F+m'M$ has self intersection 0 and is primitive. Hence, it corresponds to the fiber of a genus 1 fibration, whose smooth fibers are infinitely many genus 1 multisections of $\pi$, of degree $(-c'F+m'M)F=m'm$.
\end{proof}

 We observe that if $m=3$, i.e. $X$ is a K3 surface admitting a genus 1 fibration with a trisection, $X$ can be polarized with three non-isometric lattices that are $\left[\begin{array}{cc}0&3\\3&2c\end{array}\right]$, with $c=-1,0,1$. They correspond to the cases in which the trisection $M$ in the previous proof has genus 0,1 or 2, respectively. When $m$ grows, more possibilities for the polarizing lattices appear, see \cite[Proposition 3.7]{vanGeemen}.

For low values of $m$, we give nice geometric models of the K3 surfaces and the genus 1 fibration (see \cite[Paragraph 5.2]{vanGeemen}): let $M$ be a smooth rational multisection of degree $m$ of the genus 1 fibration whose class of the fiber is $F$, then $H:=F+M$ is an ample divisor, $\varphi_{|H|}(X)\subset\mathbb{P}^m$ and there is a pencil of hyperplanes in $\mathbb{P}^m$ containing the degree $m-2$ curve $C:=\varphi_{|H|}(M)$. 

The fibers of the genus 1 fibration are given by the residual curves of the intersection of $\varphi_{|H|}(X)$ with hyperplanes in the pencil of hyperplanes through $C$, indeed $H-C\sim F+M-M=F$. Specifically,
\begin{itemize}
\item if $m=3$, $\varphi_{|H|}(X)$ is a quartic in $\mathbb{P}^3$ containing a line $C$. The genus 1 fibration is given by the pencil of planes through $C$, which is the trisection.
\item if $m=4$, $\varphi_{|H|}(X)$ is a complete intersection of type $(2,3)$ in $\mathbb{P}^4$ containing a smooth rational curve $C$ of degree 2. There is a pencil of $\mathbb{P}^3$'s through $C$ and each $\mathbb{P}^3$ in the pencil intersects the surface at $C$ and a genus 1 curve $E$. These two curves meet in 4 points (since $E\sim (H-C)\sim  F$, $C\sim M$ and $FM=4$): the pencil of residual genus 1 curves gives a genus 1 fibration on $X$, whose multisection of degree 4 is $C$.
\item if $m=5$, $\varphi_{|H|}(X)$ is a complete intersection of type $(2,2,2)$ in $\mathbb{P}^5$ containing a smooth rational curve $C$ of degree 3. There is a pencil of $\mathbb{P}^4$'s containing the curve $C$, and each $\mathbb{P}^4$ cuts on the surface the curve $C$ and a residual genus 1 curve $E$. These two curves meets in 5 points (since $E\sim (H-C)\sim F$, $C\sim M$ and $FM=5$): the pencil of residual genus 1 curves gives a genus 1 fibration on $X$, whose multisection of degree 5 is $C$.
\end{itemize}
In all the cases above, the singular fibers of the fibration correspond to hyperplanes which cut on $\varphi_{|H|}(X)$ the curve $C$ and a singular curve. Since $M$ is a multisection of the fibration, we have a map $M\ra \mathbb{P}^1$ to the base of the fibration. It is branched over fibers cut out by hyperplanes whose intersection with $\varphi_{|H|}(X)$ consists of $C=\varphi_{|H|}(M)$ and a curve which intersects $C$ in at least one point with multiplicity higher than 1. For a general K3 surface, the hyperplanes corresponding to singular fibers and the ones corresponding to branch fibers for $M\ra \mathbb{P}^1$ do not coincide and hence the multisection $M$ is generically a saliently ramified multisection.

Given a smooth genus 1 multisection $M$ of degree $m>2$, the class $H=F+M$ is still a very ample divisor, which embeds $X$ in $\mathbb{P}^m$, and one has that the image is contained in a quadric of rank 4, as explained in \cite[Paragraph 5.3]{vanGeemen}.\\

We now specialize to the case of elliptic fibrations. For general $m$, the following holds.

\begin{theorem}\label{theorem:Picard3 higher m}
Let $X$ be a K3 surface with $\rho(X)=3$ and $X\in\mathcal{L}_d$. Then $X$ admits an elliptic fibration $\pi: X \ra \mathbb{P}^1$ such that 
\begin{itemize}
\item
$\pi$ admits a smooth rational multisection of degree $m$ if, and only if, $d\equiv 1\mod m$ and $d\geq m^2+1$;
\item $\pi$ admits a smooth genus 1 multisection of degree $m$ if, and only if, $d\equiv 0\mod m$ and $d\geq m^2$.
\end{itemize}
\end{theorem}
 \proof Similarly to the proofs of Proposition \ref{prop:polarization rational} and Lemma \ref{lemma: family Fg1x}, it suffices to look for an isometry between the lattice $\langle F,\O, M\rangle$ and a lattice $\Lambda_d$ for a certain integer $d$, where $M$ is the class of an irreducible curve with $FM=m$, $\mathcal{O}M=x\geq 0$ and $M^2=-2$, or $M^2=0$ accordingly to the genus of the multisection. More explicitly, if $M^2=-2$, then $\langle F,\O,M\rangle\simeq \langle F,\O, (-2m-x)F-m\mathcal{O}+M\rangle\simeq U\oplus \langle-2(m^2+mx+1)\simeq \Lambda_{(m^2+mx+1)}$, which does not admit even finite index overlattices in which both $\langle F,\mathcal{O}\rangle$ and $\langle (-2m-x)F-m\mathcal{O}+M\rangle$ are primitively embedded. This implies that $NS(X)\simeq \Lambda_{(m^2+mx+1)}$, which forces $d\equiv 1\mod m$ and $d\geq m^2+1$ since $x\geq 0$. The argument is analogous if $M^2=0$, by observing that $\langle F,\O,M\rangle\simeq \langle F,\O, (-2m-x)F-2\mathcal{O}+M\rangle\simeq \Lambda_{(m^2+mx)}$.\endproof

If $X$ and $\pi$ are as above, $x=\mathcal{O}M$ and $m=FM$, then the previous models for $m=3,4,5$, specialize as follows:
the surface $\varphi_{|H|}(X)\subset\mathbb{P}^m$ contains a rational curve $\varphi_{|H|}(\mathcal{O})$ of degree $x+1$ that is a section of the fibration and intersects the curve $C=\varphi_{|H|}(M)$ in $x$ points.

In Theorem \ref{theorem:Picard3}, we find smooth bisections of low genus (0 or 1) on the K3 surfaces in all the families $\mathcal{L}_{d}$ with $d\geq 4$; the smooth multisections of low genus and degree $m>2$ appear for higher values of $d$. So we do not find smooth saliently ramified multisections of degree $m>2$ on elliptic fibrations that do not admit smooth saliently ramified bisections. We obtain the following result. 
\begin{corollary}\label{cor: no multisections for low d}
The elliptic fibrations on a surface $X$ such that $\Pic(X)\simeq \Lambda_d$ admit neither rational nor genus 1 smooth multisections if and only if $d=1,2,3$.\end{corollary}

\section{Rank jumps for K3 surfaces in $\mathcal{L}_d$}\label{section: rank jump}

\begin{convention} In what follows, $\pi: X\rightarrow \mathbb{P}^1$ is an elliptic fibration on a K3 surface $X$. The \textit{field of definition of $\pi$} is the field over which  the genus 1 fibration $\pi$ admits a section. We denote it by $k$.
\end{convention}
We combine the results of Subsections \ref{subsection: base change}, \ref{subsec:Picard3} and \cite{PastenSalgado} to answer Question 1.

For the sake of completeness, we also deal with generic elliptic K3 surfaces with a smooth bisection of genus 1, i.e., in $\mathcal{L}_d$ for even $d$ (Theorem \ref{thm: rank jump_even}), and we state a result for the general K3 surfaces in the family $\mathcal{L}_1$, that is a consequence of \cite{BT}. As discussed in Subsection \ref{subsec:Picard3}, the elliptic fibration on a K3 surface general member of $\mathcal{L}_1$ admits a reducible fiber and has trivial Mordell--Weil group. By \cite[Theorem 1.1]{BT}, its rational points are potentially dense. In particular, the following holds.
\begin{theorem}\label{teorema:aut_finito}
Let $X$ be a generic K3 in the family $\mathcal{L}_1$. Let $k$ be the field of definition of the elliptic fibration on $X$. Then there is a finite extension $l/k$ such that 
$\mathcal{R}(X,\pi, l)$ is infinite. In other words, $\pi$ has the potential rank jump property.
\end{theorem}

\begin{theorem}\label{thm: rank jump odd d}
Let $(X, \pi)$ be a generic elliptic K3 surface with Picard number 3, i.e. $X$ is generic in a family $\mathcal{L}_d$.  Let $k$ be the field of definition of $\pi$, $d$ odd and $d>3$. 
\begin{itemize}
\item 
If $d\equiv 1\mod 4$, then there is an extension $l/k$ of degree at most two such that the set 
\[
\mathcal{R}(X,\pi, l)=\{t\in \mathbb{P}^1(l); \mathrm{rank} X_t(l)>\mathrm{rank} \MW(X,\pi,l)\}
\] is infinite. 
In particular, $X$ has the potential rank jump property.
\item If $d\equiv 3\mod 4$ then the set 
\[
\mathcal{R}(X,\pi, k)=\{t\in \mathbb{P}^1(k); \mathrm{rank} X_t(k)> \mathrm{rank} \MW(X,\pi,k)\}
\] is infinite. 
In particular,   $X$ has the rank jump property. 
\end{itemize}
\end{theorem}

\begin{proof}
By Theorem \ref{theorem:Picard3}, under the assumption that $d>3$ is odd,  for a generic $X$  the fibration $\pi$ admits a smooth salient bisection $B$ that is of genus 0. By Lemma \ref{lem:indep_multisections}, $B$ yields a section that is linearly independent from the sections induced by $\pi$ on the base change.

Therefore, the only statement which has not yet been proved concerns the field of definition $k$ (resp. $l$), namely, to show that the sets $\mathcal{R}(X,\pi, k)$ (resp. $\mathcal{R}(X,\pi, l)$) are infinite if $d\equiv 3\mod 4$ (resp. $d\equiv 1\mod 4$). In particular will  verify that $B$ is defined over $k$ if $d$ is odd, that it has infinitely many $k$-points if $d\equiv 3\mod 4$ and that it has infinitely many $l$-points for a certain field extension $l$ of degree two if $d\equiv 1\mod 4$. 

\textbf{Claim 1:} The field of definition of the rational bisection from Theorem \ref{theorem:Picard3} is $k$. 

Indeed, under the assumption that $d$ is odd, a generic K3 surface in a family $\mathcal{L}_d$ with Picard number 3 admits a realization as the double cover of $\mathbb{P}^2_{k}$ branched on a sextic with a unique node (Proposition \ref{prop: Fratx cover of P2}). The genus 1 fibration is defined by the pull-back of the pencil of lines through the node and thus the branch sextic and its node are defined over the same field $k$. The exceptional divisor above the node is isomorphic to a rational curve defined over $k$ and is a bisection of $\pi$. In particular, the bisection is defined over the same field as $\pi$. 

In particular, since $X$ is a double cover of the blow up of $\mathbb{P}^2_{k}$, and $B$ is defined over $k$, it follows that $B$ has infinitely many $l$-points for a certain degree two extension $l/k$.

\textbf{Claim 2:} The bisection $B$ from Theorem \ref{theorem:Picard3} has infinitely many points over $k$ if $d\equiv 3\mod 4$.

Since $B$ is rational and defined over $k$, it suffices to show that there is at least one point of $B$ defined over $k$ to conclude that there are infinitely many.

The Galois group $\Gal(l/k)$ acts on the points of $B$ by permuting the ones that are only defined over $l$ (and not over $k$) and fixing the ones defined on $k$. If $d\equiv 3 \mod 4$, then there is an odd number of points in $B\mathcal{O}$, by Remark \ref{rem; intersection B and sections}. Both the curves $B$ and $\mathcal{O}$ are defined over $k$, hence the Galois action preserves both of them. In particular it preserves the set of their intersection points. Since $\Gal(l/k)$ has order 2, if it acts on an odd number of points, then it fixes at least one of them, so there is a point in $B\cap \mathcal{O}$ (and hence on $B$) which is defined over $k$. Therefore $B$ can be embedded as a conic in the plane. The presence of a $k$-point allows us to take the projection away from it and conclude that $B$ is isomorphic to $\mathbb{P}^1$ over $k$.
\end{proof}

\begin{remark}
Theorem \ref{thm: rank jump odd d} is of particular interest in the cases $d=5,7,13$ since in all the other cases the K3 surface $X$ admits at least two distinct genus 1 fibrations which yield potential density of rational points, and by \cite[Theorem 1.1]{PastenSalgado}, it has the potential rank jump property. Even for the other values of $d\equiv 3\mod 4$,  in this level of generality, the result is stronger than the ones present in the literature so far since it shows the rank jump property over the field of definition of the elliptic fibration, without the need of a further finite field extension.
\end{remark}
\begin{remark}\label{rem: non splitting of the bisection}
Let $X\in\mathcal{L}_{d}$ with $d\equiv 1 \mod 2$ and $d\geq 5$. If $\rho(X)>3$, then $X$  admits an elliptic fibration that specializes $\pi$. The proof of the previous theorem goes through as long as the bisection $B$ in Theorem \ref{theorem:Picard3} remains a bisection and does not split in the sum of two sections. If $B$ splits, then the map $\varphi_{|F+B|}$ exhibits $X$ as double cover of $\mathbb{P}^2$ branched on a sextic that admits a singularity that is worse than a simple node, cf. Remark \ref{rem: the case d=3, spltting of the bisection}. We discuss some of cases with $\rho(X)>3$ in Section \ref{seubsect: codimension 1 subfamilies}.
\end{remark}

\begin{theorem}\label{thm: rank jump_even}
Let $(X, \pi)$ be a generic elliptic K3 surface with Picard number 3, i.e. $X$ is generic in a family $\mathcal{L}_d$.  Let $k$ be the field of definition of $\pi$. If $d$ is even and $d>2$ then there is an extension $l/k$ of degree at most two such that the set 
\[
\mathcal{R}(X,\pi, l)=\{t\in \mathbb{P}^1(l); \mathrm{rank} X_t(l)>\mathrm{rank} \MW(X,\pi,l)\}
\] is infinite. 
In particular,   $X$ has the potential rank jump property.
\end{theorem}

\begin{proof}

By Theorem \ref{theorem:Picard3}, under the assumption that $d>2$ is even, for a generic $X$ the fibration $\pi$ admits a smooth salient bisection $B$ that is of genus 1. By Lemma \ref{lem:indep_multisections}, $B$ yields a section that is linearly independent of the sections induced by $\pi$ on the base change. 

Since $X$ is generic in $\mathcal{L}_d$, the elliptic  fibration $\pi$ is not isotrivial (see Section \ref{subsubsec: K3 low Picard}). Therefore, the base change by a multisection  is a non-trivial elliptic fibration with an elliptic base and one can apply Silverman's Specialization theorem (\cite[Theorem C]{Silverman}) to the base change. In particular, the fibers of $\pi$ above the points $t$ in the image of $\pi|_B$ verify $r_t>r$. 

Therefore,
the only statement which has not yet been proved concerns the field extension $l/k$ over which $B$ has infinitely many $l$-rational points and thus over which we have infinitely many $t$'s such that $r_t>r$.  

\textbf{Claim:} There exists a field extension $l/k$ of degree at most 2 such that the genus 1 bisection(s) from Theorem \ref{theorem:Picard3} have infinitely many $l$-rational points.

If $d\geq 4$ is even then $X$ admits a realization as a double cover of $\mathbb{P}^1\times \mathbb{P}^1$ branched over a curve of bidegree $(4,4)$. 
In this case, the genus 1 fibration is given by the pull-back of the pencil of curves of bidegree $(0,1)$ (resp. $(1,0)$) and there is a pencil of bisections, of which $B$ is a member, given as the pull-back of the curves of bidegree $(1,0)$, (resp. $(0,1)$). Both pencils are clearly defined over the same field $k$. Moreover, the generic fibers of the two genus 1 pencils intersect in two geometric points that are (potentially) conjugate under $\Gal(\bar{k}(t)/k(t))$ and defined over an extension $l(t)/k(t)$ of degree at most 2. In particular, $B$ intersects each fiber $\pi^{-1}(t)$ in two $\Gal(\bar{k}(t)/k(t))$-conjugate points $p(t), p'(t)$. Fix $l=l(t)$, for a  $t$ such that $p(t)$ is of infinite order in the Mordell--Weil group of $B$. As the bisections belong to an infinite family, we can conclude that infinitely of them are elliptic curves with a positive rank over $l$, proving the claim.
\end{proof} 
\begin{remark}
The final statement in Theorem \ref{thm: rank jump_even} follows from \cite[Theorem 1.1]{PastenSalgado} or, alternatively, by combining \cite[Theorem 1.1]{GM} and \cite[Th\'eor\`eme 1.1]{CT}. In fact, both approaches imply a stronger result, namely that the set $\mathcal{R}(X,\pi, l)$ is not thin in $\mathbb{P}^1_l$.
An advantage of our point of view is that by providing a geometric realization of both genus 1 fibrations (Prop. \ref{prop: Fratx cover of P1xP1}), we are able to control the field extension $l/k$.
\end{remark}

\subsection{Rank jump properties for codimension one subfamilies}\label{seubsect: codimension 1 subfamilies}
 We discuss certain specializations of the surfaces considered above, in particular the ones which have Picard number 4 and appear in the list given by \cite{GM}, since for these the potential rank jump property is still unknown.

First we observe that if a K3 surface has an elliptic fibration and Picard number 4, its Picard group is $U\oplus \Gamma$ with base $F,\O, b_3,b_4$, and with respect to the fibration whose fiber is $F$ one of the following three cases appears: \begin{itemize}\item[a)] $\Gamma$ does not contain vectors of self-intersection $-2$: in this case the elliptic fibration has no reducible fibers and the Mordell--Weil group has rank 2;
\item[b)] $\Gamma$ contains a class of self-intersection $-2$ and the root lattice of $\Gamma$ has rank 1: in this case there is exactly one reducible fiber, which has two components, and the Mordell--Weil group has rank 1; \item[c)] $\Gamma$ is generated by roots: in this case the elliptic fibration has either two reducible fibers, each with two components, or it has one reducible fiber with three components. In both the cases the Mordell--Weil group is trivial.\end{itemize}

Since in case $c)$ the Mordell--Weil group is trivial, we now focus on the other two cases.

In case $a)$, $\Gamma$ is of type $\left[\begin{array}{cc}-2\alpha&\beta\\ \beta&-2\gamma\end{array}\right]$ with $\alpha>1$ and $\gamma>1$. The Mordell--Weil group is isomorphic to $\Z\times \Z$ and it contains the classes \begin{equation}\label{eq: sections rkMw=2}P_{n,m}:=(\alpha m^2-\beta m n+\gamma n^2)F+\O+mb_3+nb_4\end{equation} where the isomorphism between the groups is given by $P_{(n,m)}\mapsto (n,m) $ and the section $\O$ corresponds to $(0,0)$.

In case $b)$ there are two possibilities, namely that the generator of the Mordell--Weil group intersects the same component of the reducible fiber  as the section $\O$, or that it intersects in the other component. If the former holds, the Mordell--Weil lattice (see \cite[Section 11]{ScS} for definition) has an integer discriminant and the lattice $\Gamma$ is isometric to $\langle-2\rangle\oplus \langle -2\gamma\rangle$ with $\gamma>1$; in this case the Mordell--Weil group is $\Z$ and the classes of the sections are $P_n=\gamma n^2 F+\O+nb_4$.
If the latter holds, the Mordell--Weil lattice has a non-integer discriminant and the lattice $\Gamma$ is isometric to $\left[\begin{array}{cc}-2&-1\\-1&-2\gamma\end{array}\right]$ with $\gamma>1$; in this case the Mordell--Weil group is $\Z$; the classes of the sections corresponding to even integers are $P_{-2n}=(4\gamma-1) n^2 F+\O+nb_3-2nb_4$ and the ones corresponding to odd integers are $P_{-2n-1}=\left((4\gamma-1) n^2+(4\gamma-1)n+\gamma\right) F+\O+nb_3+(-2n-1)b_4$.\\

We now consider the following special choice for $\Gamma$:  $\Gamma\simeq \langle -2d\rangle \oplus\langle -2c\rangle$, i.e. $U\oplus \Gamma\simeq \Lambda_{d}\oplus \langle -2c\rangle$.
\begin{corollary}\label{cor: rank jump general rank 4}
Let $X$ be a K3 surface such that $\NS(X)\simeq \Lambda_{d}\oplus \langle -2c\rangle$ with $c\geq 1$ and $d\geq 5$, $d\equiv 3\mod 4$. Then the elliptic fibration defined on $X$ by the class of $F$ (with the notation as above) has the rank jump property.
\end{corollary}
\proof This follows as in Theorem \ref{thm: rank jump odd d}, by applying Remark \ref{rem: non splitting of the bisection} to the class $B=\frac{d+3}{2}F+2\O+b_3$, which is such that $B^2=-2$, $BF=2>0$ and $B\mathcal{O}\equiv 1\mod 2$. It is effective and,  in the light of Remark \ref{rem: non splitting of the bisection}, we now show that it cannot be written as the sum with non-negative coefficients of  sections and components of fibers.  %The result is true both if $c=1$ and hence the Mordell--Weil group of the fibration has rank 1 and if $c>1$ and hence the Mordell--Weil group of the fibration has rank 2. 
In what follows, we give the details only in the case $c>1$, which corresponds to case a) in \ref{seubsect: codimension 1 subfamilies} above. In that case, the Mordell--Weil rank is 2 and all fibers are irreducible. Assume, by contradiction, that $B$ can be written as a sum of fibers and sections. This implies that there are $r,x,y\in\mathbb{N}$ such that $B=rF+xP_{m,n}+yP_{p,q}$. 
By the condition $BF=2$, it follows $x+y=2$; by the comparison of the coefficients of $b_3$ (resp. $b_4$) in $B=\frac{d+3}{2}F+2\O+b_3$ and in $B=rF+xP_{m,n}+yP_{p,q}$, one obtains $xm+yp=1$ (resp. $xn+yq=0$). Therefore $x=y=1$, $p=1-m$ and $q=-n$. By \eqref{eq: sections rkMw=2}, $P_{m,n}+P_{1-m,-n}=(2dm^2+2cn^2-2md+d)F+2\O+b_3$  with $d\geq 5$. The coefficient of $F$ is larger than $(d+3)/2$ for every choice of $m$ and $n$, hence $r$ is forced to be a negative number, yielding the desired contradiction.

The case where $c=1$, implies the existence of a vector of self-intersection $-2$ which is orthogonal to the class of a fiber, yielding a fiber component. This corresponds to case b) in \ref{seubsect: codimension 1 subfamilies}. One obtains a contradiction in a similar way: one considers the class $B=\frac{d+3}{2}F+2\O+b_3$ and one shows that it is not possible that $B$ can be also expressed as $rF+xP_{m}+yP_{n}+zb_3$ with $r,x,y,z\geq 0$. One has to be careful since the expressions of $P_m$ and $P_n$ in terms of the basis $\{F,\mathcal{O},b_3,b_4\}$, depends on the parity of $m$ and $n$, so one has to consider all the possible cases (i.e. $m$ and $n$ are both even, are both odd, one is odd and one is even). \endproof

We now restrict our attention to the families for which the potential rank jump property is not implied by the presence of a second genus 1 fibration, i.e. the ones in the list provided in \cite{GM}.

\begin{proposition}
The K3 surfaces whose Picard lattice is isometric to one of the lattices in Table \ref{table 1} admit a unique elliptic fibration which has the potential rank jump property. 
\begin{table}[h!]
$$\begin{array}{|c|c|c|c|}
\hline
&&&\\

U\oplus \left[\begin{array}{rr}-6&-1\\-1&-6\end{array}\right]
&
U\oplus \left[\begin{array}{rr}-4&0\\0&-10\end{array}\right]
&
%U\oplus \left[\begin{array}{rr}-4&-2\\-2&-6\end{array}\right]
%&
 U\oplus \left[\begin{array}{rr}-4&0\\0&-6\end{array}\right]%\\
%&&&\\
%U\oplus \left[\begin{array}{rr}-4&-1\\-1&-4\end{array}\right]
%&
&
U\oplus \left[\begin{array}{rr}-2&-1\\-1&-10\end{array}\right]
\\
&&&\\
U\oplus \left[\begin{array}{rr}-2&-1\\-1&-6\end{array}\right]
&
U\oplus \left[\begin{array}{rr}-2&-1\\-1&-4\end{array}\right]
&
U\oplus \left[\begin{array}{rr}-2&1\\1&-2\end{array}\right]
&
U\oplus \left[\begin{array}{rr}-2&0\\0&-2\end{array}\right]\\
&&&\\
\hline\end{array}
$$

\end{table}
\begin{table}[h!]
\caption{Intersection form for Picard lattices of rank 4 of K3s with a unique elliptic fibration known to have the potentially rank jump property}
\label{table 1}
\end{table}
\end{proposition}
\proof All the K3 surfaces in the statement admit a unique elliptic fibration by \cite{BM} and have Picard number 4. We now show that in all the cases but the last two, they are special members of the families $\mathcal{L}_d$ with $d\geq 5$ and $d\equiv 1\mod 2$ and that their elliptic fibration admits a bisection $B$ that satisfies Remark \ref{rem: non splitting of the bisection}.  This  proves the statement for all the cases but the last two, which will be considered at the end of the proof.

We consider three different situations, according to the rank of the Mordell--Weil group.

{\bf First case: the Mordell--Weil group has rank 2 (case $a)$)}. This assumption corresponds to the first three lattices in the Table \ref{table 1}.

We first identify an effective class $B$ such that $BF=2$ and $B^2=-2$, and then we prove that it is supported on  an irreducible curve (in particular it is not the sum of two sections and possibly some fibers). This guarantees that it indeed corresponds to an irreducible bisection and satisfies Remark \ref{rem: non splitting of the bisection}.

For all the lattices considered here, one can prove that
there exists a base $\{b_3, b_4\}$ of $\Gamma$ such that the bilinear form is $\left[\begin{array}{rr}-2\alpha&\beta\\ \beta&-2\gamma\end{array}\right]$, with $\gamma>0$, and $\alpha\geq 5$, $\alpha\equiv 1\mod 2$. In particular $\left[\begin{array}{cc}-6&-1\\ -1&-6\end{array}\right]\simeq \left[\begin{array}{cc}-10&-5\\ -5&-6\end{array}\right]$ and 
$\left[\begin{array}{rr}-4&0\\ 0&-6\end{array}\right]\simeq \left[\begin{array}{cc}-10&-4\\ -4&-4\end{array}\right]$.

Now one argues as in Corollary \ref{cor: rank jump general rank 4}, by considering the element $B:=\frac{\alpha+3}{2}F+2\O+b_3$:
if $B$ were a reducible curve which is the union of two sections, say $P_{m,n}$ and $P_{p,q}$, and possibly some components of fibers, then there would exist three non-negative integers $x,y,z$ such that $B=xF+yP_{m,n}+zP_{p,q}$ in $\Pic(X)$. Considering the previous equations of $B$ and $P_{m,n}$ (see \eqref{eq: sections rkMw=2}), one obtains $$xF+y\left((\alpha m^2-\beta m n+\gamma n^2)F+\O+mb_3+nb_4\right)+z\left((\alpha p^2-\beta p q+\gamma q^2)F+\O+pb_3+qb_4\right)=\frac{\alpha+3}{2}F+2\O+b_3$$
which implies $y+z=2$, $ym+zp=1$, $yn+zq=0$, by comparing the coefficients of $\O$, $b_3$ and $b_4$ respectively.
In particular, one obtains that $y=z=1$, $p=1-m$, $q=-n$ and one concludes since  $x=\frac{\alpha+3}{2}-\left(2\alpha m^2-2\beta m n+2\gamma n^2+\alpha-2m\alpha+\beta n\right)$ is a negative number for any possible choice of $(m,n)\in\Z\times \Z$. Therefore $B$ is irreducible.

{\bf Second case: the Mordell Weil group has rank 1 (case $b)$)}. The cases which appear in Table 1 are $\Gamma=\left[\begin{array}{cc}-2&-1\\-1&-2\gamma\end{array}\right]$ with $\gamma=2,3,5$. For $\gamma=2,3,5$, consider the class $B=16F+2\O+b_3-4b_4$, $B=24F+2\O+b_3-4b_4$ and $B=4F+2\O-b_4$, respectively. It has the required properties (i.e. $BF=2$, $B^2=-2$) and cannot be written as $rF+xb_3+yP_{a}+zP_b$ with $y+z=2$ and $r$ and $x$ non negative, where $b_3$ is the class of an irreducible component of the reducible fiber and $P_a$ and $P_b$ are sections. 

{\bf Third cases: the Mordell Weil group is trivial (case $c)$).} This corresponds to the last two cases in Table \ref{table 1}. The fibration of each of these surfaces  has no sections of infinite order, hence it is known to have the potential rank jump property, as mentioned in Theorem \ref{teorema:aut_finito}.
\endproof

\begin{remark}
We now consider the four rank 4 lattices listed in \cite{GM}, but not contained in Table 1.

If $\Gamma\simeq \left[\begin{array}{rr}-4&-2\\-2&-4\end{array}\right]$, then it is not possible to find in $\Gamma$ a class with self-intersection $-2d$ such that $d$ is odd, hence it is not possible to show the existence of a smooth rational bisection with the same method applied before. Moreover, since \cite{GM} guarantees that there are no more than one genus 1 fibration on the surface, the elliptic fibration does not admit a smooth genus 1 bisection. \\

If $\Gamma\simeq\left[\begin{array}{rr}-2&0\\0&-4\end{array}\right]$ in order to find a rational bisection, we need a class $B=xF+y\O+z b_3+wb_4$ such that $B^2=-2$, $BF=2$, $0\leq Bb_3\leq 2$ (since $b_3$ is identified with the class of an irreducible component of a reducible fiber). This forces $y=2$, $z=-1$ and hence $x=2+w^2$. So $B=(2+w^2)F+2\O+b_3+wb_4$, which is equal to $2F+b_3+2P_{w/2}$ if $w$ is even and to $b_3+P_{v}+P_{v-1}$ if $w=2v-1$. In both the cases $B$ corresponds to the union of sections and components of fibers, and so it does not correspond to a smooth rational bisection. So for this choice of $\Gamma$ the fibration does not admit a smooth bisection of genus either 0 or 1.

If $\Gamma\simeq\left[\begin{array}{rr}-4&\mu\\\mu&-2\lambda\end{array}\right]$ with $(\mu,\lambda)=(-2,3),(-1,2)$ we are not able to find a convenient $B$ which represents an irreducible smooth rational bisection of the fibration (but we are not claiming that such $B$ does not exist). 

Finally, we observe that at least for some of the lattices in the Table \ref{table 1}, it is not possible to find a smooth rational bisection $B$ such that $B\mathcal{O}\equiv 1\mod 2$, and hence to prove by this strategy the rank jump property, and not just the potential rank jump property. This is the case, for example, of the lattice $U\oplus \langle -4\rangle\oplus \langle -6\rangle$; indeed a class $B$ such that $B^2=-2$, $BF=2$ and $\O B\mathcal\equiv 1 \mod 2$ is necessarily of the type $B_{k,h}=(4k^2+6(h^2+h)+3)F+2\mathcal{O}+2kb_3+(2h+1)b_4$, $k,h\in\mathbb{Z}$, but then $B_{k,h}=P_{k,h}+P_{k,h+1}$, so it splits in the union of two sections.
\end{remark}

We observe that the similar strategies can be applied to consider subfamilies of $\mathcal{L}_d$ with higher codimension.

\section{Examples}\label{section: examples}

In this section, we describe several examples for which we prove the potential rank jump property and/or the rank jump property on very explicit K3 surfaces with Picard number 3, for which we give an equation. 
We give two examples of elliptic K3 surfaces with Picard rank 3 over $k=\mathbb{Q}(\sqrt{5})$, realized as a double cover of a nodal sextic with principal tangents defined only over a degree 2 extension. In Example \ref{ex:picard 3 over Q}, the bisection above the node is $\mathbb{P}^1_{\mathbb{Q}}$. In Example \ref{ex:picard3overquadratic}, it is a conic without $k$-points. Therefore, by the proof of Theorem \ref{thm: rank jump odd d}, the former admits the rank jump property over $k$, while the latter admits the property after a suitable degree 2 extension of $k$. 

In the last example we use a trisection to prove the  rank jump property on another K3 surface of Picard number 3.

In all these examples the geometric construction provides immediately the field of definition of the rational multisection and the field where it has rational points.

\begin{example}\label{ex:picard 3 over Q}
Let $C$ be the singular plane sextic curve given by the vanishing of the following polynomial:
\begin{align*}
f:=2x^4y^2  + 8x^4z^2 + x^3y^3 + x^3yz^2 + x^3z^3  + x^2y^4 + 5x^2y^2z^2 + 7x^2z^4 + xy^5+xy^4z +xy^3z^2 + 6xz^5 + 5y^4z^2.
\end{align*}

Then $C$ has a unique singularity, namely $p=(1:0:0)$, which is of type $A_1$. It admits a bitangent line $L: x=0$. As this line intersects the curve in only 2 points, namely $(0:1:0)$ and $(0:0:1)$, with intersection multiplicities 2 and 4, respectively, it can be considered as a degenerate tritangent line. In particular, it splits in the double cover $g:X\ra \mathbb{P}^2$ branched on $C$. Indeed, if a curve $D$ on a surface $Y$ intersects the branch locus of a double cover $g:X\ra Y$ with an with even multiplicity in a point $q$, then its inverse image $g^{-1}(D)\subset X$ is singular in $g^{-1}(q)$ but does not ramify in $g^{-1}(q)$ (because of the splitting of the local equation of the double cover). So, for the double cover $g:X\ra \mathbb{P}^2$ branched on $C$, the inverse image of the line $L$ does not contain ramification points but has two singular points (which are  $g^{-1}(0:1:0)$ and $g^{-1}(0:0:1)$). In particular, since $L$ is rational and so simply connected, the inverse image $g^{-1}(L)$ consists of two rational curves meeting in two points (one with multiplicity 2).

Let $X$ be the K3 double cover of $\mathbb{P}^2$ branched over $C$, then the geometric Picard number of $X$ is at least 3.  Indeed, its Picard lattice contains $B$, the rational curve above the exceptional divisor given by the desingularization of $p$, a genus 1 class $F$, given by the pull-back of any line through $p$, and the class $\O$ given by one of the two irreducible rational curves above the line $L$. Since $p\not\in L$, their intersection matrix is as follows
$$
\begin{pmatrix}-2&2&0\\
	2&0&1\\
	0&1&-2
	\end{pmatrix}
$$
and, with the same notation of Section \ref{subsec: Frat x}, we have $x=B\O=0$, which implies $d=5$, i.e., the Picard group contains a lattice isometric to $\Lambda_{5}=U\oplus \langle -10\rangle$. 

We now prove that $\rho(X)=3$, which will imply that $\Pic(X)\simeq \Lambda_{5}$ and that $L$ is the unique line in $\mathbb{P}^2$ which splits in the double cover.

 Our approach to determine the Picard number of $X$ is based on Tate's conjecture (\cite{Tate}). More precisely, we bound the geometric Picard number from above by the number of eigenvalues of the characteristic polynomial of Frobenius, for primes of good reduction, that are roots of unity (see \cite[Corollary 2.3]{RvL} for the statement and \cite[Sec. 4]{EJ} for detailed examples of the approach in a similar setting).

In our case, a computation of the characteristic polynomial of the Frobenius for ordinary primes such as $7, 11, 23, 31$ and 37, done with the help of Magma (\cite{Magma}), yields a factor of degree 18 that is not cyclotomic.  Since the characteristic polynomial of Frobenius has degree $b_2=22$, we obtain $\rho\leq 22-18= 4$.

We claim that $\rho=3$ and verify it by exhibiting a class of a smooth rational curve that does not lift to characteristic 0. 

For the prime $7$, the sextic $C$ has two singularities of type $A_1$, namely $p=(1:0:0)$ and $q=(1:5:1)$. The exceptional divisor above $q$ is a smooth rational curve. So the Picard group for $p=7$ is the lattice $\Lambda_{5}\oplus \langle -2\rangle$, where the last summand is spanned by the exceptional divisor over $q$. If this curve lifted over $\mathbb{Q}$, then the Picard group of $X$ over $\Q$ would be $\Lambda_{5}\oplus \langle -2\rangle$ and in particular the elliptic fibration induced by the pencil of lines through $p$ would have a fiber of type $I_2$. But one can directly check on the equation of the fibration (see, \ref{eq: elliptic fib}), that it does not admit any reducible fiber over $\Q$. Therefore the class which is defined for $p=7$ as exceptional divisor over $q$ does not lift over $\Q$ and this implies that $\rho=3$ over $\Q$.

Finally, we verify that the smooth bisection $B$ is saliently ramified. Recall that the unique elliptic fibration on $X$ is given by the pull-back of lines through $p$. Moreover, the restriction map $\pi_B: B\rightarrow \mathbb{P}^1$ is branched over the fibers that are pull-backs of the two lines that are the principal tangents of $C$ at the node $p$, namely $y=2iz$ and $y=-2iz$. Since these intersect $C$ transversally outside $p$, they pull-back to smooth fibers of $\pi$. 

To be more explicit, the equation of the K3 surface $X$ is $w^2=f$ and the pencil of lines through $p$ is given by $y=tz$, so the equation of the genus 1 fibration with fiber $F$ is 
\begin{equation}\label{eq: elliptic fib}
w^2=x^4(2t^2  + 8) + x^3z(t^3 + 3 + 1)  + x^2z^2(t^4 + 5t^2 + 7) + xz^3(t^5+t^4 +t^3 + 6)+ 5t^4z^4.\end{equation}
It intersects the inverse image of the line $x=0$ in the curve $w^2=5t^4z^4$ and if $\sqrt{5}$ exists in the field of definition, this curve splits in the two curves $w\pm\sqrt{5}z^2t^2=0$ and thus corresponds to the two sections $(x(t),z(t),w(t))=(0,1,\pm\sqrt{5}t^2)$. Therefore, the field of definition of the elliptic fibration (with section) is $k=\Q(\sqrt{5})$. 

The elliptic fibration has no reducible fibers, as can be checked by the computation of its discriminant.

The bisection corresponds to the exceptional divisor over $p$, i.e., to $z=0$ and thus its equation is  $w^2=x^2(2t^2+8)$ which in particular describes a rational curve defined over $\Q$. It intersects the fiber above $t=2$ at the $\mathbb{Q}$-points $(x,z,w)=(1,0,\pm 4)$, so it is isomorphic to $\mathbb{P}^1_{\mathbb{Q}}$ and has infinitely many $\mathbb{Q}$-points, and, therefore, infinitely many $k$-rational points.   

We observe that since the elliptic fibration has infinitely many sections, there are infinitely many irreducible curves in $\mathbb{P}^2$ that split in the double cover. The section $S_n$ comes from the splitting of a curve of degree $10n^2-10n+1$ with a singular point of multiplicity $10n^2+10n$ in $p$.
\end{example}

\begin{example}\label{ex:picard3overquadratic}
We consider a slight modification of the coefficients of $x^4y^2$ and $x^4z^2$ in the sextic in Example \ref{ex:picard 3 over Q}. Let $C$ be the singular plane sextic curve given by the vanishing of the following polynomial:
\begin{align*}
f:=-5x^4y^2  -6x^4z^2 + x^3y^3 + x^3yz^2 + x^3z^3  + x^2y^4 + 5x^2y^2z^2 + 7x^2z^4 + xy^5+xy^4z +xy^3z^2 + 6xz^5 + 5y^4z^2.
\end{align*}

Then, as before, $C$ has a unique singularity, namely $p=(1:0:0)$, which is of type $A_1$. It admits a degenerate tritangent line, namely $L: x=0$, which turns out to be unique, once one proved as above that $\rho=3$. 

The same calculations as in Example \ref{ex:picard 3 over Q} can be used to show that the Picard group is isometric to $\Lambda_5$. The upper bound on the geometric Picard number can be found by inspecting ordinary primes such as $p=7,13,61$. In fact, this K3 surface has the same reduction modulo 7 as in Example \ref{ex:picard 3 over Q}, so the same argument applies to show that the Picard group has indeed rank 3. 

The bisection over $k$ which corresponds to the exceptional divisor above $p$ has equation $w^2=x^2(-5t^2-6)$, which clearly admits no $\mathbb{R}$-points and, in particular, it has not $k$-points (but it has $l$-points where $l$ is a field extension of degree 2 of $k$).
\end{example}

The next example is concerned with a $\Lambda_d$-polarized K3 surface, with an even $d$. In it, the rank jump property is realized over the field of definition of the elliptic fibration and is obtained by using a trisection (and not as above a bisection) as the ones described in Section \ref{section: higher degree}. It is inspired by \cite[Theorem 3.3]{NijghvL}.

\begin{example}\label{ex:even_d}
Let $X$ be the quartic K3 surface given by 
\begin{equation}
q_1f_1+q_2f_2+q_3f_3=0
\end{equation}
\end{example}
where $f_1, f_2, f_3$ are polynomials of degree 2 in $\mathbb{Z}[x,y,z,w]$ given as follows:
\[
f_1:=x^2-xy;\, f_2:=z^2-zw; \, f_3:= yz-xw; \
\]

and $q_1, q_2, q_3$ are also of degree 2  in $\mathbb{Z}[x,y,z,w]$ and satisfy:
$$q_1 \equiv y^2 +xw+xy +yw \mod 2$$
$$q_2 \equiv xy +yz +yw+zw \mod 2,$$
$$q_3 \equiv xz +xw+z^2 +w^2 +y^2 \mod 2,$$
and,

$$q_1 \equiv y^2 +xw+2xz +2yz \mod 3,$$
$$q_2 \equiv xy +yz +yw+x^2 +2xw \mod 3,$$
$$q_3 \equiv xz +xw+z^2 +w^2 +2xy +2yw \mod 3.$$
Then, by \cite[Thm. 3.3]{NijghvL}, $X$ is a K3 surface with Picard number 3.  Its Picard lattice is generated by the classes of $H, L_0, L_1$, where $H$ corresponds to a hyperplane section, $L_0$ is the class of the line given by  $x=z=0$ and $L_1$, the class of $x-y=z-w=0$. 
Moreover, it is $U\oplus \langle -20 \rangle$-polarized. In fact, the divisor $D=6H-7L_0+3L_1$ is perpendicular to a copy of $U$ and has self-intersection $-20$. 

The surface $X$ admits two evident elliptic fibrations. Indeed, the residual intersection of a hyperplane through any line is a genus 1 curve in $X$. Therefore, the  pencil of hyperplanes through one of the lines $L_0$ or $L_1$ yields an elliptic fibration, which we denote by $\pi_0, \pi_1$, in reference to taking hyperplanes through $L_0$, or $L_1$, respectively, in what follows.

As a general hyperplane through $L_0$ (resp. $L_1$) does not contain $L_1$ (resp. $L_0$), the line $L_1$ (resp. $L_0$) corresponds to a section of the elliptic fibration $\pi_0$ (resp. $\pi_1$). On the other hand, the line $L_0$ (resp. $L_1$) is a trisection of $\pi_0$ (resp. $\pi_1$). From the construction, the fibers of $\pi_0$ are smooth genus 1 bisections of $\pi_1$ and vice versa.

We therefore have an example of elliptic fibrations that admit a rational trisection and a (family of) genus 1 bisection(s). 

We claim that $\pi_0$ and $\pi_1$ have the rank jump property over $\mathbb{Q}$. To verify this claim, it suffices to show that the rational trisection of each fibration is saliently ramified. Alternatively, one can show that at least one of the infinitely many saliently ramified bisections of genus 1 has infinitely many $\mathbb{Q}$-points.

Let us choose $q_i$ as follows
 $$q_1 = y^2 +xw+3xy +3yw +2xz+2yz,$$
$$q_2 = xy +yz +yw+4x^2+2xw+3zw,$$
$$q_3= xz +xw+z^2 +w^2 +3y^2 +2xy+2yw.$$
Consider the line $L_0$ and the pencil of planes $x=tz$ through that line. 
The intersection of the quartic with $x=tz$ contains the line $L_0$; the residual cubic is the fiber of the elliptic fibration $\pi_0$ over the point $t$. One can directly check that for $t=0$ the intersection between the line and the residual cubic consists of the points $(0:0:0:1)$ and $(0:1:0:0)$ and the first one has multiplicity two. In particular, the trisection (which corresponds to $L_0$) branches on the fiber over $t=0$. Nevertheless, one can also check that the intersection $X\cap \{x=0\}$ consists of the line $L_0$ and of a smooth cubic, so the trisection $L_0$ branches over a smooth fiber of the fibration, i.e. it is saliently ramified.

\section*{Acknowledgements}
It is our pleasure to thank the referee for their helpful comments and for pointing out a mistake in the previous version of the paper, Bert van Geemen for reading and commenting on our manuscript, and Wim Nijgh for pointing out the quartic K3 that appears in Example \ref{ex:even_d} and discussing details about it.
The second author also would like to thank Brendan Hassett for a discussion on the ramification of trisections on elliptic K3 surfaces and the \textit{Institute for Advanced Study} for their hospitality and stimulating atmosphere. Part of her work on this material is supported by the NWO XL grant \textit{Rational points: new dimensions}, the James D. Wolfensohn Fund and the National Science Foundation under Grant No. DMS-1926686. The first author is member of the INDAM-GNSAGA.

\end{document}